\documentclass[11pt, a4paper,regno]{amsart}
\usepackage{amsmath, amsthm, amsfonts, amssymb}
\usepackage{a4wide}

\usepackage[latin1]{inputenc}

\usepackage{nicefrac}

\usepackage{enumitem}

\DeclareMathOperator{\N}{\mathbb{N}}

\numberwithin{equation}{section}
	\newtheorem{satz}{Satz}[section]
	\newtheorem{thm}[satz]{Theorem}
	\newtheorem{lemma}[satz]{Lemma}
	 
	\newtheorem{cor}[satz]{Corollary}
	
	\theoremstyle{definition} 
	\newtheorem{defin}[satz]{Definition}
	
	\newtheorem{rem}[satz]{Remark}

\begin{document}

\parindent 0cm

\address{Fachbereich Mathematik, Technische Universit\"at Darmstadt, Schlossgartenstrasse 7, D--64289 Darmstadt, Germany}

\author{David Klein and Jennifer Kupka}
\email{klein@mathematik.tu-darmstadt.de}
\email{jenny.kupka@t-online.de}

\title[Completions and algebraic formulas for mock theta functions]{Completions and algebraic formulas for the coefficients of Ramanujan's mock theta functions}

\date{\today}

\maketitle

\begin{abstract}
We present completions of mock theta functions to harmonic weak Maass forms of weight $\nicefrac{1}{2}$ and algebraic formulas for the coefficients of mock theta functions. We give several harmonic weak Maass forms of weight $\nicefrac{1}{2}$ that have mock theta functions as their holomorphic part. Using these harmonic weak Maass forms and the Millson theta lift we compute finite algebraic formulas for the coefficients of the appearing mock theta functions in terms of traces of singular moduli.
\end{abstract}

\mbox{}

\section{Introduction} \label{section introduction}
Mock theta functions first appeared in Ramanujan's last letter to his friend Hardy in 1920. In this letter he told Hardy that he had discovered a new class of functions which he called mock theta functions. Ramanujan did not give any definition of what a mock theta function should be, but listed 17 examples, divided into four groups of orders $3$, $5$, $7$ and $10$, respectively, given as $q$-hypergeometric series, and stated various identities between them and some analytical properties. For example, the four mock theta functions of order $3$ that Ramanujan defined in his letter are
\begin{flalign*}
f(q) &:= \sum_{n=0}^{\infty} \frac{q^{n^{2}}}{\left(-q;q \right)_{n}^{2}}, \ \ \ \ \ \ \ \ \ \ \
\phi(q) := \sum_{n=0}^{\infty} \frac{q^{n^{2}}}{\left(-q^{2};q^{2} \right)_{n}}, \\
\psi(q) &:= \sum_{n=1}^{\infty} \frac{q^{n^{2}}}{\left(q;q^{2} \right)_{n}}, \ \ \ \ \ \ \ \ \ \ \ \
\chi(q) := \sum_{n=0}^{\infty} \frac{q^{n^{2}} \left(-q;q \right)_{n}}{\left(-q^{3};q^{3} \right)_{n}}, 
\end{flalign*}

where we have used the standard notation
$$\left(a;q^{k} \right)_{n} := \prod_{m=0}^{n-1} \left(1-a q^{mk} \right).$$

Since then many mathematicians (especially Watson in his work \cite{watson1936final}) have dealt with Ramanujan's 17 functions, and have proven many of the identities he had given. A number of 16 further mock theta functions were later found in Ramanujan's Lost Notebook (see, e.g., \cite{ramanujan1988lost} and \cite{andrews2005ramanujan}), including seven functions of order $6$. Other mathematicians have also discovered more mock theta functions that had not been considered before: In \cite{gordon2000some} Gordon and McIntosh found functions of order $8$ while McIntosh also studied mock theta functions of order $2$ in \cite{mcintosh2007second}. \\
Articles that offer a good first overview on this topic are, for example, \cite{zagier2009ramanujan} and \cite{duke2014almost}. A more detailed survey over all mock theta functions of the different orders, including their definitions, relations and transformation formulas is provided in \cite{gordon2012survey}. In this paper we will use the standard definitions of the mock theta functions as given in \cite{gordon2012survey}. \\
One major breakthrough in a deeper understanding of mock theta functions came in 2002 when Sander Zwegers found a connection between mock theta functions and harmonic weak Maass forms of weight $\nicefrac{1}{2}$. He proved that a mock theta function could be completed to a harmonic weak Maass form of weight $\nicefrac{1}{2}$ by multiplying it by a suitable power of $q$ and subsequently adding a certain non-holomorphic function to it. Zwegers considered these completions for the fifth and seventh order mock theta functions in his PhD thesis \cite{zwegers2008mock}, and for two of the third order mock theta functions in \cite{zwegers2001mock}. Moore followed the work of Zwegers and found transformation laws for mock theta functions of order $10$ and their relation to harmonic weak Maass forms in \cite{moore2012modular}. Though Ramanujan had not explained what the order of a mock theta function should be, it turned out that the order is related to the level of the corresponding Maass form. \\
We will present such completions to a harmonic weak Mass form of weight $\nicefrac{1}{2}$ for $22$ different mock theta functions of orders $2$, $3$, $6$ and $8$. For example, we will show for the sixth order mock theta function
$$\sigma(q) := \sum_{n=0}^{\infty} \frac{q^{\frac{1}{2} (n+1) (n+2)} \left(-q;q \right)_{n}}{\left(q;q^{2} \right)_{n+1}}$$

that the function $q^{-\frac{1}{12}} \ \sigma(q)$ is the holomorphic part of a harmonic weak Maass form of weight $\nicefrac{1}{2}$ for the subgroup
$$\{(\gamma,\phi) \in \text{Mp}_{2}(\mathbb{Z}) \ | \ \gamma \in \Gamma(6) \}$$

of the metaplectic group $\text{Mp}_{2}(\mathbb{Z})$, where $\Gamma(6)$ is the principal congruence subgroup of level $6$. \\
A further example of what we will prove is that, if
$$F(\tau) = \left(\begin{matrix}
  f_{0}(\tau) \\
  f_{1}(\tau)  \\
  f_{2}(\tau)  \\
  f_{3}(\tau)  \\
  f_{4}(\tau)  \\
  f_{5}(\tau) 
\end{matrix}\right) := \left(\begin{matrix}
  \sqrt{8} \ q^{-\frac{1}{12}} \ \sigma(q) \\
  2 \ q^{\frac{1}{4}} \ \rho(q) \\
  q^{-\frac{1}{48}} \ \phi(q^{\frac{1}{2}}) \\
  q^{-\frac{1}{48}} \ \phi(-q^{\frac{1}{2}}) \\
  \sqrt{2} \ q^{-\frac{3}{16}} \ \psi(q^{\frac{1}{2}}) \\
  \sqrt{2} \ q^{-\frac{3}{16}} \ \psi(-q^{\frac{1}{2}}) 
\end{matrix}\right)$$

with $q:=e^{2 \pi i \tau}$, $\tau \in \mathbb{H}$, and the mock theta functions $\sigma$, $\rho$, $\phi$ and $\psi$ of order $6$, then the function
\begin{flalign*}
\widetilde{F}(\tau) := &\sqrt{2} \ f_{0}(\tau) \ [-(\mathfrak{e}_{2}-\mathfrak{e}_{22})-(\mathfrak{e}_{10}-\mathfrak{e}_{14})] + 2 \ f_{1}(\tau) \ [-(\mathfrak{e}_{6}-\mathfrak{e}_{18})] \\
&+ (f_{2}(\tau)+f_{3}(\tau)) \ [(\mathfrak{e}_{1}-\mathfrak{e}_{23})-(\mathfrak{e}_{7}-\mathfrak{e}_{17})] + (f_{2}(\tau)-f_{3}(\tau)) \ [(\mathfrak{e}_{5}-\mathfrak{e}_{19})-(\mathfrak{e}_{11}-\mathfrak{e}_{13})] \\
&+ \sqrt{2} \ (f_{4}(\tau)+f_{5}(\tau)) \ (\mathfrak{e}_{3}-\mathfrak{e}_{21}) + \sqrt{2} \ (f_{4}(\tau)-f_{5}(\tau)) \ [-(\mathfrak{e}_{9}-\mathfrak{e}_{15})],
\end{flalign*}
where $\mathfrak{e}_r$ are the standard basis vectors of the group algebra $\mathbb{C}[\mathbb{Z}/24\mathbb{Z}]$, is the holomorphic part of a harmonic weak Maass form of weight $\nicefrac{1}{2}$ for the dual Weil representation. This result opens up the possibility to use the powerful tool of theta lifts between spaces of modular forms.\\
The Millson theta lift, which maps weight 0 to weight $\nicefrac{1}{2}$ harmonic weak Maass forms, uses the Millson theta function as an integration kernel and was studied in great detail by Alfes in her thesis \cite{alfesthesis} and by Alfes-Neumann and Schwagenscheidt in \cite{alfes2018theta}. In particular, Alfes-Neumann found formulas for the coefficients of the holomorphic part of the Millson theta lift in terms of traces of singular moduli. By writing the harmonic weak Maass form of weight $\nicefrac{1}{2}$ containing the mock theta functions as the Millson theta lift of a suitable weakly holomorphic modular form, we can derive finite algebraic formulas for the coefficients of the considered mock theta functions in terms of traces of singular moduli. Continuing our example from above, we will prove that the coefficients $a_\sigma(n)$ of the mock theta function $\sigma$ of order 6 are given by
$$a_\sigma(n)=-\frac{i}{4\sqrt{48n-4}}\big(\textup{tr}^+_{e_{(6),1}}(4-48n,2)-\textup{tr}^-_{e_{(6),1}}(4-48n,2)\big), $$
where the trace functions $\textup{tr}^+_{e_{(6),1}}$ and $\textup{tr}^-_{e_{(6),1}}$ are given as in (\ref{definition trace functions}), and $e_{(6),1}\in M^!_0(12)$ is defined as
$$e_{(6),1}(z):=\bigg(\frac{\eta(z)\eta(3z)}{\eta(4z)\eta(12z)}\bigg)^2-16\bigg(\frac{\eta(4z)\eta(12z)}{\eta(z)\eta(3z)}\bigg)^2$$
with $\eta(\tau) = q^{\frac{1}{24}} \ \prod_{n=1}^{\infty} (1-q^{n})$ denoting the Dedekind eta function. Similar formulas for the order 3 mock theta functions $f$ and $\omega$ (see, e.g., \cite{bruinier2017algebraic} for its definition) have already been proven by Bruinier and Schwagenscheidt in \cite{bruinier2017algebraic}. \\
This paper is organized as follows. We will start with the necessary definitions, notations and results in section 2, followed by the results on the completions and formulas for the coefficients of the mock theta functions in section 3. We will consider mock theta functions of different order separately and subsection 3.1 about those of order $6$ will be worked out in detail. As the ideas and strategies for the other orders are very similar to the case of order $6$, the subsections about the other orders only contain the known results and no proofs. \\
Most of the results presented in this paper first appeared in our Master's theses \cite{kupka2017mock} and \cite{klein2018mock} where they also have been proven in more detail. 

\mbox{}

\textbf{Acknowledgement} \\
Both our theses were supervised by Jan Hendrik Bruinier and Markus Schwagenscheidt. We thank them for their support and many helpful discussions during the writing of our theses, and also for their comments on this paper. Further, we thank Kathrin Bringmann and Anna-Maria von Pippich for their helpful remarks.

\mbox{} 

\section{Preliminaries}\label{section preliminaries}
\subsection{Lattices, the Weil Representation and theta functions} \label{subsection lattices and the weil representation}

Let $N>0$ be an integer. We consider the lattice $L=\mathbb{Z}$ with the quadratic form $n \mapsto N n^{2}$. The discriminant group $\mathcal{D}:=L'/L$ can then be identified with $\mathbb{Z}/2N\mathbb{Z}$ together with the $\mathbb{Q}/\mathbb{Z}$-valued quadratic form $r\mapsto \frac{r^2}{4N} \ (\text{mod} \ \mathbb{Z})$. The associated bilinear form on $\mathcal{D}$ is $(r,r')=\frac{rr'}{2N} \ (\text{mod} \ \mathbb{Z})$. \\
For $r\in L'/L$ we define $\mathfrak{e}_r$ to be the standard basis vectors of the
group algebra $\mathbb{C}[L'/L]$ equipped with the standard inner product
$\langle\cdot,\cdot\rangle$ satisfying
$\langle\mathfrak{e}_r,\mathfrak{e}_{r'}\rangle=\delta_{r,r'}$. The associated \emph{Weil representation} $\rho_L$ is defined on the generators $T=\big((\begin{smallmatrix}1&1\\0&1
\end{smallmatrix}),1 \big)$ and $S=\big((\begin{smallmatrix}0&-1\\1&0
\end{smallmatrix}),\sqrt{\tau} \big)$ of the metaplectic group
$\text{Mp}_{2}(\mathbb{Z})$ by 
\begin{align}\label{WeilDarstellung}
\rho_L(T)\mathfrak{e}_r=e\big(Q(r)\big)\mathfrak{e}_r\ \ \text{ and }\ \
\rho_L(S)\mathfrak{e}_r=\frac{e(-1/8)}{\sqrt{2N}}\sum_{r'(2N)}e\big(-(r,r')\big)\mathfrak{e}_{r'},
\end{align}

where $e(z)=e^{2\pi iz}$ for $z\in\mathbb{C}$ and $\sqrt{z}=z^{\frac{1}{2}}$ always denotes the principal branch of the square root. The dual Weil representation corresponds to the lattice $L$ with quadratic form $-Q$ and will be denoted by $\overline{\rho}_L$. \\
Let $N$ be as above and $a \in \mathbb{Z}$. For $\tau \in \mathbb{H}$ we define the \emph{unary theta function $\theta_{N}$ of level $N$}  as
$$\theta_{N}(\tau) := \sum_{a \ (2N)} \theta_{N,a} (\tau) \ \mathfrak{e}_{a}, \ \
\text{ where } \ \
\theta_{N,a}(\tau) := \sum_{n \equiv a \ (2N)} n \ q^{\frac{n^{2}}{4N}} = \sum_{n \equiv a \ (2N)} n \ e^{2 \pi i \tau \frac{n^{2}}{4N}}.$$

The definition of $\theta_{N,a}$ depends only on $a \ (2N)$. If we consider the lattice above as well as its associated Weil representation, then the vector valued theta function $\theta_{N}$ is a holomorphic vector valued modular form of weight $\nicefrac{3}{2}$ for this Weil representation. Thus, the function $\theta_{N,a}$ is holomorphic on $\mathbb{H}$ and has the modular transformation properties
\begin{equation} \label{transformation of theta_N,a under T}
\theta_{N,a}(\tau+1) = e \left(\frac{a^{2}}{4N} \right) \ \theta_{N,a}(\tau)
\end{equation}

and
\begin{equation} \label{transformation of theta_N,a under S}
\theta_{N,a} \left(-\frac{1}{\tau} \right) = \tau^{\frac{3}{2}} \ \frac{e \left(-\frac{1}{8} \right)}{\sqrt{2N}} \ \sum_{k \ (2N)} e \left(-\frac{ak}{2N} \right) \ \theta_{N,k}(\tau).
\end{equation}

Let $Q$ be an exact divisor of $N$, i.e. $Q\in\mathbb{Z}_{>0}$ with $Q\vert N$ and
$\gcd(N/Q,Q)=1$. The \emph{Atkin-Lehner involution} associated to $Q$ is then defined by
any matrix $$W_Q^N=\begin{pmatrix}
Q\alpha&\beta\\N\gamma&Q\delta
\end{pmatrix}, $$ where $\alpha,\beta,\gamma,\delta\in\mathbb{Z}$ with
$\det(W_Q^N)=Q$. The map $$W_Q^N:\ M_k(N)\rightarrow M_k(N),\ \ f\mapsto f|_kW_Q^N $$
does not depend on the choice of $\alpha,\beta,\gamma$ and $\delta$ and defines an
involution. For two exact divisors $Q,Q'$ of $N$ we define the product $$Q\ast
Q':=\frac{Q\cdot Q'}{\gcd(Q,Q')^2},$$ which is compatible with the action of the
Petersson slash operator, i.e. we have $$f|_kW_{Q\ast
Q'}^N=f|_kW_Q^N|_kW_{Q'}^N.$$ 
The automorphism group $\text{Aut}(\mathbb{Z}/2N\mathbb{Z})$ acts on vector valued
modular forms $f=\sum_{r\in\mathbb{Z}/2N\mathbb{Z}}f_r \ \mathfrak{e}_r$ for $\rho_L$ or
$\overline{\rho}_L$ by $$f^\sigma=\sum_{r}f_r \ \mathfrak{e}_{\sigma(r)}.$$ These
automorphisms are all involutions, which are also called Atkin-Lehner involutions
and correspond to exact divisors $Q$ of $N$. The automorphism $\sigma_Q$
corresponding to $Q$ is defined by the two equations \begin{align*}
\sigma_Q(r)\equiv-r \ (2Q)\ \text{and}\ \sigma_Q(r)\equiv r \ (2N/Q)
\end{align*} for an element $r\in\mathbb{Z}/2N\mathbb{Z}$.

\subsection{Harmonic Maass Forms and the $\xi$-Operator} \label{subsection harmonic maass forms and the xi-operator}

Vector valued harmonic weak Maass forms were first introduced by Bruinier and Funke in \cite{bruinier2004two}. We will consider a more general setting than they have in their article. \\
Let $V$ be a vector space over $\mathbb{C}$ of finite dimension $d$ and let $k \in \frac{1}{2} \mathbb{Z}$ with $k \neq 1$. For $\tau \in \mathbb{H}$ we put $u:=\text{Re}(\tau)$ and $v:=\text{Im}(\tau)$, so that $\tau=u+iv$. Moreover, recall the weight $k$ hyperbolic Laplace operator, given by
$$\Delta_{k} = -v^{2} \ \left(\frac{\partial^{2}}{\partial u^{2}} + \frac{\partial^{2}}{\partial v^{2}} \right) + ikv \ \left(\frac{\partial}{\partial u} + i \ \frac{\partial}{\partial v} \right).$$

Let $\rho: \text{Mp}_{2}(\mathbb{Z}) \rightarrow \text{GL}(V)$ be a unitary representation of $\text{Mp}_{2}(\mathbb{Z})$ that satisfies $\rho(T)^{N}=\text{Id}$ for some $N \in \mathbb{N}$, let $f: \mathbb{H} \rightarrow V$ be a twice continuously differentiable function and $\Gamma \subseteq \text{Mp}_{2}(\mathbb{Z})$ a subgroup of finite index. We call $f$ a \emph{harmonic weak Maass form of weight} $k$ with respect to the representation $\rho$ and the group $\Gamma$ if
\begin{enumerate}
    \item $f(\gamma \tau)=\phi(\tau)^{2k} \ \rho(\gamma,\phi) \ f(\tau)$ for all $(\gamma,\phi) \in \Gamma$,
    \item there is a constant $C>0$ such that for any cusp $s \in \mathbb{Q} \cup \{\infty\}$ of $\Gamma$ and $(\delta,\phi) \in \text{Mp}_{2}(\mathbb{Z}$) with $\delta \infty=s$ the function $f_{s}(\tau):=\phi(\tau)^{-2k} \ \rho^{-1}(\delta,\phi) \ f(\delta \tau)$ satisfies $f_{s}(\tau)=O(e^{Cv})$ as $v \rightarrow \infty$ (uniformly in $u$),
    \item $\Delta_{k} f=0$.
\end{enumerate}

Condition ii) says that $f$ increases at most linear exponentially at all cusps of $\Gamma$.

The space of these forms is denoted by $H_{k,\rho}(\Gamma)$. If we have $\Gamma=\text{Mp}_{2}(\mathbb{Z})$, we write as an abbreviation $H_{k,\rho}(\text{Mp}_{2}(\mathbb{Z}))=:H_{k,\rho}$. Further, let $M^{!}_{k,\rho}$ be its subspace of weakly holomorphic modular forms, consisting of those forms in $H_{k,\rho}$ that are holomorphic on $\mathbb{H}$. \\
A harmonic weak Maass form $f \in H_{k,\rho}$ has a unique decomposition $f = f^{+}+f^{-}$, where $f^{+}$ is the \emph{holomorphic part} and $f^{-}$ is the \emph{non-holomorphic part of} $f$. 
If we write the Fourier expansion of the holomorphic part of $f \in H_{k,\rho}$ as 
$$f^{+}(\tau) = \sum_{n \in \mathbb{Z}} a^{+}(n) \ e \left(\frac{n \tau}{N} \right),$$

where $a^{+}(n)$ are vector valued coefficients, then the Fourier polynomial
$$P(f)(\tau) = \sum_{n \in \mathbb{Z}, n \leq 0} a^{+}(n) \ e \left(\frac{n\tau}{N} \right)$$

is called the \emph{principal part of} $f$. \\
For $f \in H_{k,\rho}$ the differential operator $\xi_{k}$ is given by
$$\xi_{k}(f)(\tau) = 2i \ v^{k} \ \overline{\frac{\partial}{\partial \overline{\tau}} f(\tau)}.$$

The operator $\xi_{k}$ is antilinear and defines a surjective mapping $\xi_{k}: H_{k,\rho} \rightarrow M_{2-k,\overline{\rho}}^{!}$ with kernel given by $M_{k,\rho}^{!}$. We can use $\xi_{k}$ to define the subspace
$$H_{k,\rho}^{+} := \{f \in H_{k,\rho} \ | \ \xi_{k}(f) \in S_{2-k,\overline{\rho}} \},$$

so that $H_{k,\rho}^{+}$ consists of all harmonic weak Maass forms in $H_{k,\rho}$ that are mapped to cusp forms under $\xi_{k}$. The holomorphic part $f^{+}$ of $f \in H_{k,\rho}^{+}$ is sometimes also called a \emph{mock modular form}, and $\xi_{k}f$ is called the \emph{shadow} of $f$. \\
We will use the following lemma when we prove our formulas for the coefficients.
\begin{lemma}[\cite{bruinier2017algebraic}, Lemma 2.3]\label{XiSpitze} Let $G$ be a harmonic weak Maass form of weight $2-k\in\nicefrac{1}{2}+\mathbb{Z}$ for $\rho_L$ or $\overline{\rho}_L$ whose principal part vanishes and which maps to a cusp form under $\xi_{2-k}$ (or a holomorphic modular form if $k=\nicefrac{1}{2}$). Then $G$ is a cusp form.
\end{lemma}

\subsection{The Millson Theta Lift and Traces of CM-Values}
For a discriminant $D<0$ and $r\in\mathbb{Z}$ with $D\equiv r^2 \ (4N)$ denote by
$\mathcal{Q}_{N,D,r}$ the set of integral binary quadratic forms
$Q(x,y)=ax^2+bxy+cy^2$ of discriminant $D=b^2-4ac$ and satisfying $N\vert a$ and
$b\equiv r \ (2N)$. This set splits into the sets of positive and negative definite
quadratic forms, which we denote by $\mathcal{Q}^+_{N,D,r}$ and
$\mathcal{Q}^-_{N,D,r}$, respectively. The group $\Gamma_0(N)$ acts on both of these
sets with finitely many orbits and the number $\omega_Q=\frac{1}{2}\vert
\Gamma_0(N)_Q\vert$ is finite. For each $Q\in\mathcal{Q}^+_{N,D,r}$ the equation
$Q(z_Q,1)=0$ is solved by the associated \emph{CM-point} $z_Q=(-b+i\sqrt{\vert
D\vert})/2a$. 
\\For a weakly holomorphic modular form $F\in M_0^!(N)$ of weight $0$ for $\Gamma_0(N)$ we define the two \emph{trace
functions}
\begin{align} \label{definition trace functions}
\text{tr}_F^+(D,r)=\sum_{Q\in\mathcal{Q}^+_{N,D,r}/\Gamma_0(N)}\frac{F(z_Q)}{\omega_Q}\ \
 \text{ and }\ \
\text{tr}_F^-(D,r)=\sum_{Q\in\mathcal{Q}^-_{N,D,r}/\Gamma_0(N)}\frac{F(z_Q)}{\omega_Q}.
\end{align}
The \emph{Millson theta lift} $\mathcal{I}^M(F,\tau)$ of a weakly holomorphic modular form 
$F\in M_0^!(N)$ is defined as an integral \begin{align*}
\mathcal{I}^M(F,\tau)=\frac{i}{\sqrt{N}}\int_{\Gamma_0(N)\setminus\mathbb{H}}F(z)\ \Theta_M(\tau,z)\ \frac{dxdy}{y^2},
\end{align*} where we write $z=x+iy$ and $\Theta_M(\tau,z)$ denotes the Millson theta
function. The theta function $\Theta_M(\tau,z)$ is $\Gamma_0(N)$-invariant in the
variable $z$ and transforms like a modular form of weight $\nicefrac{1}{2}$ for the dual Weil
representation $\overline{\rho}_L$ in the variable $\tau$. The assignment
$F\mapsto\mathcal{I}^M(F,\tau)$ then defines a map $\mathcal{I}^M:\
M_0^!(N)\rightarrow H_{1/2,\overline{\rho}_L}$. For more details see \cite{alfesthesis} or \cite{alfes2018theta}. As it turns out, the coefficients of
the holomorphic part of the Millson theta lift can be computed using the trace
functions which we defined above. 
\begin{thm}[\cite{alfesthesis}, Theorem 4.3.1]\label{Alfes}
Let $F\in H_0^+(N)$ be a harmonic weak Maass form of weight $0$ for $\Gamma_0(N)$,
$D<0$ a discriminant and $r\in L'/L$ with $D\equiv r^2 \ (4N)$. Then the coefficient
of index $(-D,r)$ of the holomorphic part of the Millson theta lift
$\mathcal{I}^M(\tau,F)$ is given by 
\begin{align*}
\frac{i}{\sqrt{-D}}\big(\textup{tr}_F^+(D,r)-\textup{tr}_F^-(D,r)\big).
\end{align*}
\end{thm}

\mbox{} 

\section{Completions and algebraic formulas for the coefficients of mock theta functions}
\subsection{Mock Theta Functions of order $6$} \label{section mock theta functions of order 6}

We want to complete sixth order mock theta functions to harmonic weak Maass forms and want to derive algebraic formulas for their coefficients. For this aim we will first construct two different vector valued Maass forms, one containing the sixth order functions $\sigma, \rho, \phi$ and $\psi$ and the other comprising $\mu, \lambda, \nu$ and $\xi$. Their definitions, and also the definitions of the mock theta functions of other orders, can be found in \cite{gordon2012survey}. Afterwards we will derive the transformation behaviour of its components. Starting from our vectors we will further construct two vector valued harmonic weak Maas forms for the dual Weil representation. We will then be able to obtain algebraic formulas for the coefficients of the mentioned mock theta functions.

\begin{defin}
For $\tau \in \mathbb{H}$ we define the vector valued functions
$$F_{(6),1}(\tau) := \left(\begin{matrix}
  \sqrt{8} \ q^{-\frac{1}{12}} \ \sigma(q) \\
  2 \ q^{\frac{1}{4}} \ \rho(q) \\
  q^{-\frac{1}{48}} \ \phi(q^{\frac{1}{2}}) \\
  q^{-\frac{1}{48}} \ \phi(-q^{\frac{1}{2}}) \\
  \sqrt{2} \ q^{-\frac{3}{16}} \ \psi(q^{\frac{1}{2}}) \\
  \sqrt{2} \ q^{-\frac{3}{16}} \ \psi(-q^{\frac{1}{2}}) 
\end{matrix}\right) \ \ \ \ \ \text{and} \ \ \ \ \
F_{(6),2}(\tau) := \left(\begin{matrix}
  -\sqrt{2} \ q^{-\frac{1}{12}} \ \mu(q) \\
  - q^{\frac{1}{4}} \ \lambda(q) \\
  -2 \ q^{-\frac{1}{48}} \ \nu(q^{\frac{1}{2}}) \\
  -2 \ q^{-\frac{1}{48}} \ \nu(-q^{\frac{1}{2}}) \\
  -\sqrt{8} \ q^{-\frac{3}{16}} \ \xi(q^{\frac{1}{2}}) \\
  -\sqrt{8} \ q^{-\frac{3}{16}} \ \xi(-q^{\frac{1}{2}}) 
\end{matrix}\right)$$

with $q=e^{2 \pi i \tau}$.
\end{defin}

These two functions have the same modular transformation properties as the following lemma states.

\begin{lemma} \label{lemma transformation properties of F_(6),1}
For $j=1,2$ and $\tau \in \mathbb{H}$ the function $F_{(6),j}$ satisfies
\begin{equation} \label{transformation of F_(6),1 under T}
F_{(6),j}(\tau+1) = \left(\begin{matrix}
  \zeta_{12}^{-1} & 0 & 0 & 0 & 0 & 0 \\
  0 & i & 0 & 0 & 0 & 0 \\
  0 & 0 & 0 & \zeta_{48}^{-1} & 0 & 0 \\
  0 & 0 & \zeta_{48}^{-1} & 0 & 0 & 0 \\
  0 & 0 & 0 & 0 & 0 & \zeta_{16}^{-3} \\
  0 & 0 & 0 & 0 & \zeta_{16}^{-3} & 0
\end{matrix}\right) \ F_{(6),j}(\tau)
\end{equation}

and
\begin{equation} \label{transformation of F_(6),1 under S}
\frac{1}{\sqrt{-i \tau}} \ F_{(6),j} \left(-\frac{1}{\tau} \right) = \left(\begin{matrix}
  0 & 0 & \frac{1}{\sqrt{3}} & 0 & \sqrt{\frac{2}{3}} & 0 \\
  0 & 0 & \sqrt{\frac{2}{3}} & 0 & -\frac{1}{\sqrt{3}} & 0 \\
  \frac{1}{\sqrt{3}} & \sqrt{\frac{2}{3}} & 0 & 0 & 0 & 0 \\
  0 & 0 & 0 & \frac{1}{\sqrt{3}} & 0 & -\sqrt{\frac{2}{3}} \\
  \sqrt{\frac{2}{3}} & -\frac{1}{\sqrt{3}} & 0 & 0 & 0 & 0 \\
  0 & 0 & 0 & -\sqrt{\frac{2}{3}} & 0 & -\frac{1}{\sqrt{3}}
\end{matrix}\right) \ F_{(6),j}(\tau) + R_{(6)}(\tau),
\end{equation}

where
$$R_{(6)}(\tau) := \frac{\sqrt{6}i}{\tau} \ \left(\begin{matrix}
  -\sqrt{8} \ J_{1}(\frac{6 \pi i}{\tau}) \\
  -2 \ J(\frac{6 \pi i}{\tau}) \\
  J_{1}(\frac{3 \pi i}{2 \tau}) \\
  K_{1}(\frac{3 \pi i}{\tau}) \\
  \frac{1}{\sqrt{2}} \ J(\frac{3 \pi i}{2 \tau}) \\
  \sqrt{2} \ K(\frac{3 \pi i}{\tau})
\end{matrix}\right),$$

and $J, J_{1}, K, K_{1}$ are given by
\begin{flalign*}
J(\alpha) \ &= \ \int_{0}^{\infty} \frac{e^{-\alpha x^{2}}}{\cosh(\alpha x)} \ dx, \ \ \ \ \ \ \ \ \ \ \ \ \ \ \ \ \ \ \
K(\alpha) \ = \ \int_{0}^{\infty} e^{-\frac{1}{2} \alpha x^{2}} \ \frac{\cosh \left(\frac{1}{2} \alpha x \right)}{\cosh(\alpha x)} \ dx, \\
J_{1}(\alpha) \ &= \ \int_{0}^{\infty} e^{-\alpha x^{2}} \ \frac{\cosh \left(\frac{2}{3} \alpha x \right)}{\cosh(\alpha x)} \ dx, \ \ \ \ \ \ \
K_{1}(\alpha) \ = \ \int_{0}^{\infty} e^{-\frac{1}{2} \alpha x^{2}} \ \frac{\cosh \left(\frac{5}{6} \alpha x \right) - \cosh \left(\frac{1}{6} \alpha x \right)}{\cosh(\alpha x)} \ dx.
\end{flalign*}

\end{lemma}

\begin{proof}
Let $j=1$. The formula (\ref{transformation of F_(6),1 under T}) follows directly if we insert $\tau+1$. \\
If we use the transformation formulas for $\sigma(q)$, $\rho(q)$, $\phi(-q)$ and $\psi(-q)$ in \cite{gordon2012survey}, p.~123 with $\alpha=\nicefrac{3 \pi i}{\tau}$ (which implies $q=e^{-\nicefrac{3 \pi i}{\tau}}$, $\beta=-\nicefrac{\pi i \tau}{3}$ and $q_{1}=e^{\nicefrac{2 \pi i \tau}{6}}$), as well as the formulas for $\phi(q)$ and $\psi(q)$ with $\alpha=\nicefrac{3 \pi i}{2 \tau}$ (which yields $q=e^{-\nicefrac{3 \pi i}{2 \tau}}$, $\beta=-\nicefrac{2 \pi i \tau}{3}$ and $q_{1}=e^{\nicefrac{2 \pi i \tau}{3}}$), we obtain (\ref{transformation of F_(6),1 under S}). \\
For $j=2$ the proof is analogous, using the transformation formulas for $\mu$, $\lambda$, $\nu$ and $\xi$.

\end{proof}

We can now write the function $R_{(6)}$ from the previous lemma in terms of integrals over sums of theta functions $\theta_{N,a}$ which have been defined in subsection \ref{subsection lattices and the weil representation}.

\begin{lemma} \label{lemma integral representation of R_(6)}
For $\tau \in \mathbb{H}$ we have
\begin{equation} \label{integral representation of R_(6)}
R_{(6)}(\tau) = \frac{i^{\frac{3}{2}}}{\sqrt{24}} \ \int_{0}^{i \infty} \frac{g_{(6)}(z)}{\sqrt{-i(z \tau -1)}} \ dz,
\end{equation}

where $g_{(6)}$ is the vector $(g_{(6),0},g_{(6),1},g_{(6),2},g_{(6),3},g_{(6),4},g_{(6),5})^{T}$ and
\begin{flalign*}
g_{(6),0}(z) &:= \sqrt{2} \ (\theta_{12,2}(z)+\theta_{12,10}(z)), \\
g_{(6),1}(z) &:= 2 \ \theta_{12,6}(z), \\
g_{(6),2}(z) &:= -(\theta_{12,1}(z)+\theta_{12,5}(z)-\theta_{12,7}(z)-\theta_{12,11}(z)), \\
g_{(6),3}(z) &:= -(\theta_{12,1}(z)-\theta_{12,5}(z)-\theta_{12,7}(z)+\theta_{12,11}(z)), \\
g_{(6),4}(z) &:= -\sqrt{2} \ (\theta_{12,3}(z)-\theta_{12,9}(z)), \\
g_{(6),5}(z) &:= -\sqrt{2} \ (\theta_{12,3}(z)+\theta_{12,9}(z)).
\end{flalign*}

\end{lemma}

The integration over a vector valued function in the lemma means that we integrate each of its components.

\begin{proof}
Let
$$M_{(6)} := \left(\begin{matrix}
  0 & 0 & \frac{1}{\sqrt{3}} & 0 & \sqrt{\frac{2}{3}} & 0 \\
  0 & 0 & \sqrt{\frac{2}{3}} & 0 & -\frac{1}{\sqrt{3}} & 0 \\
  \frac{1}{\sqrt{3}} & \sqrt{\frac{2}{3}} & 0 & 0 & 0 & 0 \\
  0 & 0 & 0 & \frac{1}{\sqrt{3}} & 0 & -\sqrt{\frac{2}{3}} \\
  \sqrt{\frac{2}{3}} & -\frac{1}{\sqrt{3}} & 0 & 0 & 0 & 0 \\
  0 & 0 & 0 & -\sqrt{\frac{2}{3}} & 0 & -\frac{1}{\sqrt{3}}
\end{matrix}\right).$$

Replacing $\tau$ by $-\nicefrac{1}{\tau}$ in the transformation formula for $S$ and subsequently multiplying both sides by $\frac{1}{\sqrt{-i \tau}} \ M_{(6)}$ yields
$$R_{(6)}(\tau)=-\frac{1}{\sqrt{-i \tau}} \ M_{(6)} \ R_{(6)} \left(-\frac{1}{\tau} \right).$$

If we choose $\tau:=it$ with $t \in \mathbb{R}$, $t>0$, we get
$$R_{(6)}(it) = -\frac{1}{\sqrt{t}} \ M_{(6)} \ R_{(6)} \left(\frac{i}{t} \right).$$

We consider the first component
$$\sqrt{6t} \ \left(-\frac{1}{\sqrt{3}} \ J_{1} \left(\frac{3 \pi t}{2} \right) - \frac{1}{\sqrt{3}} \ J \left(\frac{3 \pi t}{2} \right) \right)$$

of this vector. If we use the identity $J_{1}(\alpha)=\frac{1}{2} \ J(\alpha) + \frac{1}{6} \ J(\frac{\alpha}{9})$ (see, e.g., \cite{gordon2012survey}, p.~122), the partial fraction decomposition
$$\frac{1}{\cosh(\pi y)} = -\frac{i}{\pi} \ \sum_{n \in \mathbb{Z}} \frac{1}{y-i \left(2n+\frac{1}{2} \right)}-\frac{i}{\pi} \ \sum_{n \in \mathbb{Z}} \frac{1}{-y-i \left(2n+\frac{1}{2} \right)}$$ 

and the identity 
$$\int_{-\infty}^{\infty} \frac{e^{-\pi t y^{2}}}{y-ir} \ dy = \pi i r \ \int_{0}^{\infty} \frac{e^{-\pi r^{2} u}}{\sqrt{u+t}} \ du$$

for $r \in \mathbb{R}$, $r \neq 0$ and $t \in \mathbb{R}$, $t>0$ (see, e.g., \cite{zwegers2001mock}, Lemma 1.18), then a straightforward computation yields
\begin{flalign*}
&\sqrt{6t} \ \left(-\frac{1}{\sqrt{3}} \ J_{1} \left(\frac{3 \pi t}{2} \right) - \frac{1}{\sqrt{3}} \ J \left(\frac{3 \pi t}{2} \right) \right) \\
&= \ \frac{2 i^{\frac{3}{2}}}{\sqrt{3} \sqrt{i t}} \ \int_{0}^{i \infty} \left(\frac{3 \ \sum_{n \in \mathbb{Z}} \left(2n+\frac{1}{2} \right) \ e^{6 \pi i (2n+1/2)^{2} z}}{\sqrt{-i \left(z-\frac{1}{it} \right)}} + \frac{\sum_{n \in \mathbb{Z}} \left(2n+\frac{1}{2} \right) \ e^{\frac{2}{3} \pi i (2n+1/2)^{2} z}}{\sqrt{-i \left(z-\frac{1}{it} \right)}} \right) \ dz.
\end{flalign*}

The identity above is valid for all $t \in \mathbb{R}$, $t>0$, thus, the identity theorem for holomorphic functions yields that for all $\tau \in \mathbb{H}$ the first component of $R_{(6)}(\tau)$ is equal to
$$\frac{2}{\sqrt{3}} \ i^{\frac{3}{2}} \ \int_{0}^{i \infty} \frac{3 \ \sum_{n \in \mathbb{Z}} \left(2n+\frac{1}{2} \right) \ e^{6 \pi i (2n+1/2)^{2} z} + \sum_{n \in \mathbb{Z}} \left(2n+\frac{1}{2} \right) \ e^{\frac{2}{3} \pi i (2n+1/2)^{2} z}}{\sqrt{-i (z \tau - 1)}} \ dz.$$

To rewrite the numerator in terms of theta functions we note that
$$\sum_{n \equiv 2 \ (3)} \left(2n+\frac{1}{2} \right) \ e^{\frac{2}{3} \pi i (2n+1/2)^{2} z} = -3 \cdot \sum_{n \in \mathbb{Z}} \left(2n+\frac{1}{2} \right) \ e^{6 \pi i (2n+1/2)^{2} z}.$$

By a calculation this implies
$$3 \ \sum_{n \in \mathbb{Z}} \left(2n+\frac{1}{2} \right) \ e^{6 \pi i (2n+1/2)^{2} z} + \sum_{n \in \mathbb{Z}} \left(2n+\frac{1}{2} \right) \ e^{\frac{2}{3} \pi i (2n+1/2)^{2} z} = \frac{1}{4} \ (\theta_{12,2}(z)+\theta_{12,10}(z)).$$

Hence the first component of identity (\ref{integral representation of R_(6)}) follows. \\
Using the appropriate partial fraction decompositions of the appearing functions the identities for the other components can be verified analogously. For more details we refer the reader to \cite{klein2018mock}.
\end{proof}

Now we can define a non-holomorphic function $G_{(6)}$ such that $F_{(6),1}-G_{(6)}$ and $F_{(6),2}-G_{(6)}$ are vector valued harmonic weak Maass forms.

\begin{defin}
For $\tau \in \mathbb{H}$ let
$$G_{(6)}(\tau) := \frac{i}{\sqrt{24}} \ \int_{-\overline{\tau}}^{i \infty} \frac{g_{(6)}(z)}{\sqrt{-i(z+\tau)}} \ dz,$$

with $g_{(6)}$ as defined in Lemma \ref{lemma integral representation of R_(6)}.
\end{defin}

\begin{lemma}
The function $G_{(6)}$ has the same modular transformation properties under $\tau \mapsto \tau+1$ and $\tau \mapsto -\nicefrac{1}{\tau}$ as the one of $F_{(6),1}$ and $F_{(6),2}$, stated in Lemma \ref{lemma transformation properties of F_(6),1}.
\end{lemma}

\begin{proof}
Let
$$N_{(6)} := \left(\begin{matrix}
  \zeta_{12}^{-1} & 0 & 0 & 0 & 0 & 0 \\
  0 & i & 0 & 0 & 0 & 0 \\
  0 & 0 & 0 & \zeta_{48}^{-1} & 0 & 0 \\
  0 & 0 & \zeta_{48}^{-1} & 0 & 0 & 0 \\
  0 & 0 & 0 & 0 & 0 & \zeta_{16}^{-3} \\
  0 & 0 & 0 & 0 & \zeta_{16}^{-3} & 0
\end{matrix}\right).$$

We use formula (\ref{transformation of theta_N,a under T}) with $z$ replaced by $z-1$ and obtain
$$g_{(6)}(z-1) = N_{(6)} \ g_{(6)}(z).$$

This leads to the identity
$$G_{(6)}(\tau+1) = N_{(6)} \ G_{(6)}(\tau)$$

by a transformation of the defining integral. \\
Using formula (\ref{transformation of theta_N,a under S}) we get the transformation behaviour
$$g_{(6)} \left(-\frac{1}{z} \right) = (-iz)^{\frac{3}{2}} \ (-M_{(6)}) \ g_{(6)}(z).$$

Via an integral transformation this gives us the identities
$$\frac{1}{\sqrt{-i \tau}} \ G_{(6)} \left(-\frac{1}{\tau} \right) = -\frac{i}{\sqrt{24}} \ \int_{0}^{-\overline{\tau}} \frac{M_{(6)} \ g_{(6)}(u)}{\sqrt{-i(u+\tau)}} \ du$$

and
$$\frac{1}{\sqrt{-i \tau}} \ G_{(6)} \left(-\frac{1}{\tau} \right) - M_{(6)} \ G_{(6)}(\tau) = R_{(6)}(\tau).$$

\end{proof}

Using the last lemma we now get that $F_{(6),1}$ and $F_{(6),2}$ are the holomorphic parts of two vector valued harmonic weak Maass forms of weight $\nicefrac{1}{2}$.

\begin{thm} \label{theorem transformation properties of H_(6),1, H_(6),2}
The functions $H_{(6),1}$ and $H_{(6),2}$, defined for $\tau \in \mathbb{H}$ by
\begin{flalign*}
H_{(6),1}(\tau)&:=F_{(6),1}(\tau)-G_{(6)}(\tau), \\
H_{(6),2}(\tau)&:=F_{(6),2}(\tau)-G_{(6)}(\tau),
\end{flalign*}

are vector valued harmonic weak Maass forms of weight $\nicefrac{1}{2}$ for the metaplectic group $\emph{\text{Mp}}_{2}(\mathbb{Z})$. \\
For $j=1,2$ and $\tau \in \mathbb{H}$ we have
\begin{equation} \label{transformation properties of H_(6),1, H_(6),2 under T}
H_{(6),j}(\tau+1) = \left(\begin{matrix}
  \zeta_{12}^{-1} & 0 & 0 & 0 & 0 & 0 \\
  0 & i & 0 & 0 & 0 & 0 \\
  0 & 0 & 0 & \zeta_{48}^{-1} & 0 & 0 \\
  0 & 0 & \zeta_{48}^{-1} & 0 & 0 & 0 \\
  0 & 0 & 0 & 0 & 0 & \zeta_{16}^{-3} \\
  0 & 0 & 0 & 0 & \zeta_{16}^{-3} & 0
\end{matrix}\right) \ H_{(6),j}(\tau)
\end{equation}

and
\begin{equation} \label{transformation properties of H_(6),1, H_(6),2 under S}
 H_{(6),j} \left(-\frac{1}{\tau} \right) = \sqrt{-i \tau} \ \left(\begin{matrix}
  0 & 0 & \frac{1}{\sqrt{3}} & 0 & \sqrt{\frac{2}{3}} & 0 \\
  0 & 0 & \sqrt{\frac{2}{3}} & 0 & -\frac{1}{\sqrt{3}} & 0 \\
  \frac{1}{\sqrt{3}} & \sqrt{\frac{2}{3}} & 0 & 0 & 0 & 0 \\
  0 & 0 & 0 & \frac{1}{\sqrt{3}} & 0 & -\sqrt{\frac{2}{3}} \\
  \sqrt{\frac{2}{3}} & -\frac{1}{\sqrt{3}} & 0 & 0 & 0 & 0 \\
  0 & 0 & 0 & -\sqrt{\frac{2}{3}} & 0 & -\frac{1}{\sqrt{3}}
\end{matrix}\right) \ H_{(6),j}(\tau).
\end{equation}

\end{thm}

\begin{cor} \label{corollary xi-image order 6}
We have $\xi_{1/2}(H_{(6),1})(\tau)=\xi_{1/2}(H_{(6),2})(\tau)=-\frac{1}{\sqrt{12}} \ g_{(6)}(\tau)$.
\end{cor}

Now we know the transformation behaviour of the functions $H_{(6),1}, H_{(6),2}$ under the generators of the modular group as well as the explicit representations to which they transform. We will see now that we can use the transformation properties in Theorem \ref{theorem transformation properties of H_(6),1, H_(6),2} to obtain two functions that transform to the Weil representation. \\
More precisely, we consider the lattice $L$ defined at the beginning of subsection \ref{subsection lattices and the weil representation} with $N=12$, and its associated Weil representation (\ref{WeilDarstellung}). We find the following result:

\begin{lemma}
Suppose that the function $H=(h_{0},h_{1},h_{2},h_{3},h_{4},h_{5})^{T}$ satisfies the transformation properties (\ref{transformation properties of H_(6),1, H_(6),2 under T}) and (\ref{transformation properties of H_(6),1, H_(6),2 under S}) in Theorem \ref{theorem transformation properties of H_(6),1, H_(6),2}. Then the function
\begin{flalign*}
\widetilde{H} := &\sqrt{2} \ h_{0} \ [-(\mathfrak{e}_{2}-\mathfrak{e}_{22})-(\mathfrak{e}_{10}-\mathfrak{e}_{14})] + 2 \ h_{1} \ [-(\mathfrak{e}_{6}-\mathfrak{e}_{18})] \\
&+ (h_{2}+h_{3}) \ [(\mathfrak{e}_{1}-\mathfrak{e}_{23})-(\mathfrak{e}_{7}-\mathfrak{e}_{17})] + (h_{2}-h_{3}) \ [(\mathfrak{e}_{5}-\mathfrak{e}_{19})-(\mathfrak{e}_{11}-\mathfrak{e}_{13})] \\
&+ \sqrt{2} \ (h_{4}+h_{5}) \ (\mathfrak{e}_{3}-\mathfrak{e}_{21}) + \sqrt{2} \ (h_{4}-h_{5}) \ [-(\mathfrak{e}_{9}-\mathfrak{e}_{15})]
\end{flalign*}

transforms like a vector valued modular form of weight $\nicefrac{1}{2}$ for the dual Weil representation $\overline{\rho}_{L}$ considered above.
\end{lemma}

From the last lemma we immediately obtain two vector valued harmonic weak Maass forms $\widetilde{H}_{(6),1}$, $\widetilde{H}_{(6),2}$ of weight $\nicefrac{1}{2}$ for $\text{Mp}_{2}(\mathbb{Z})$ and the dual Weil representation $\overline{\rho}_{L}$ of level $N=12$, if we apply the lemma for $H=H_{(6),1}$ and $H=H_{(6),2}$, respectively. Hence $\widetilde{H}_{(6),1}, \widetilde{H}_{(6),2} \in H_{1/2,\overline{\rho}_{L}}^{+}$. \\
Now we come back to our initial functions $H_{(6),1}$ and $H_{(6),2}$ and want to relate their components to scalar valued harmonic weak Maass forms. In order to do that we consider the congruence subgroup
$$\Gamma(6) = \left\{\left(\begin{matrix}
  a & b \\
  c & d \\
\end{matrix}\right) \in \text{SL}_{2}(\mathbb{Z}) \ \Big| \ b \equiv c \equiv 0 \ (6), \ a \equiv d \equiv 1 \ (6) \right\}.$$

With the use of Sage \cite{sagemath} we determined a system of generators for this group, decomposed the generators into products of $S$ and $T$, and multiplied the corresponding matrices from Theorem \ref{theorem transformation properties of H_(6),1, H_(6),2} according to these products, to obtain the transformation properties of $H_{(6),1}$ and $H_{(6),2}$ under all generators.  All of the appearing transformation matrices are diagonal, so we get:

\begin{thm} \label{theorem order 6 scalar valued harmonic Maass forms}
For $j=1,2$ the components of the vector valued harmonic weak Maass form $H_{(6),j}$ are scalar valued harmonic weak Maass forms of weight $\nicefrac{1}{2}$ for the subgroup
$$\{(\gamma,\phi) \in \emph{\text{Mp}}_{2}(\mathbb{Z}) \ | \ \gamma \in \Gamma(6) \}$$

of the metaplectic group $\emph{\text{Mp}}_{2}(\mathbb{Z})$.
\end{thm}

Hence the sixth order mock theta functions $\sigma,\rho,\phi,\psi,\mu,\lambda,\nu$ and $\xi$ are the holomorphic parts of scalar valued harmonic weak Maass forms.

\begin{rem}
The $\xi$-images of the harmonic weak Maass forms in Theorem \ref{theorem order 6 scalar valued harmonic Maass forms} can be easily obtained from Corollary \ref{corollary xi-image order 6} by looking at the components of $\xi_{1/2}(H_{(6),1})(\tau)$ and $\xi_{1/2}(H_{(6),2})(\tau)$.
\end{rem}

As an application of the Millson theta lift we can now compute the coefficients of
the treated mock theta functions in terms of traces of singular moduli by writing them as the Millson theta lift of a suitable weakly holomorphic modular form. 
\begin{defin}
We define the functions
\begin{align}\label{e61}
e_{(6),1}(z):=\bigg(\frac{\eta(z)\eta(3z)}{\eta(4z)\eta(12z)}\bigg)^2-16\bigg(\frac{\eta(4z)\eta(12z)}{\eta(z)\eta(3z)}\bigg)^2
\end{align}
and
\begin{align}\label{e62}
e_{(6),2}(z):=\bigg(\frac{\eta(z)\eta(3z)}{\eta(4z)\eta(12z)}\bigg)^4-16^2\bigg(\frac{\eta(4z)\eta(12z)}{\eta(z)\eta(3z)}\bigg)^4.
\end{align}
\end{defin}
These functions are weakly holomorphic modular forms of weight $0$, level $12$ whose principal parts
start with $q^{-1}$ and $q^{-2}$, respectively. 
\begin{thm}\label{Formeln61}
Let $e_{(6),1}(z)\in M_0^!(12)$ be defined as in (\ref{e61}).
\begin{enumerate}
\item[(1)] For $n\geq0$ the coefficients $a_\sigma(n)$ of $\sigma(q)$ are given by
\begin{align*}
a_\sigma(n)&=-\frac{i}{4\sqrt{48n-4}}\big(\textup{tr}^+_{e_{(6),1}}(4-48n,2)-\textup{tr}^-_{e_{(6),1}}(4-48n,2)\big).
\end{align*}
\item[(2)] For $n\geq0$ the coefficients $a_\rho(n)$ of $\rho(q)$ are given by
\begin{align*}
a_\rho(n)&=-\frac{i}{4\sqrt{48(n+1)-36}}\big(\textup{tr}^+_{e_{(6),1}}(36-48(n+1),6)-\textup{tr}^-_{e_{(6),1}}(36-48(n+1),6)\big).
\end{align*}
\item[(3)] For $n\geq0$ the coefficients $a_\phi(n)$ of $\phi(q)$ are given by
\begin{align*}
a_\phi(n)&= \begin{cases}
\frac{i}{2\sqrt{48n-1}}\big(\textup{tr}^+_{e_{(6),1}}(1-48n,1)-\textup{tr}^-_{e_{(6),1}}(1-48n,1)\big),&\text{
if }n\text{ is even,} \\
\frac{i}{2\sqrt{48n-25}}\big(\textup{tr}^+_{e_{(6),1}}(25-48n,5)-\textup{tr}^-_{e_{(6),1}}(25-48n,5)\big), &\text{
if }n\text{ is odd.} \end{cases}
\end{align*}
\item[(4)] For $n\geq0$ the coefficients $a_\psi(n)$ of $\psi(q)$ are given by
\begin{align*}
a_\psi(n)&=\begin{cases}
\frac{i}{4\sqrt{48n-9}}\big(\textup{tr}^+_{e_{(6),1}}(9-48n,3)-\textup{tr}^-_{e_{(6),1}}(9-48n,3)\big),  &\text{if }n\text{ is even,} \\
\frac{i}{-4\sqrt{48(n+1)-81}}\big(\textup{tr}^+_{e_{(6),1}}(81-48(n+1),9)-\textup{tr}^-_{e_{(6),1}}(81-48(n+1),9)\big), &\text{if }n\text{ is odd.} \end{cases}
\end{align*}
\end{enumerate}

\end{thm}
\begin{proof}
As already proven before, the function $\widetilde{H}_{(6),1}$ is a vector valued harmonic
weak Maass form of weight $\nicefrac{1}{2}$ for the dual Weil representation. Using the series
expansion of $\sigma,\rho,\phi$ and $\psi$ one immediately sees that its principal
part is given by $2 \
q^{-\frac{1}{48}} \ (\mathfrak{e}_1-\mathfrak{e}_7+\mathfrak{e}_{17}-\mathfrak{e}_{23})$.
The function $e_{(6),1}$ is an eigenfunction of all Atkin-Lehner involutions, with
eigenvalue $+1$ for the operators $W_1$ and $W_3$ and eigenvalue $-1$ for $W_4$ and
$W_{12}$. Thus, the Fourier expansions of $e_{(6),1}$ at the cusps of $\Gamma_0(12)$
only differ by a possible minus sign. Then the Millson theta lift maps the function
$e_{(6),1}$ to a harmonic weak Maass form of weight $\nicefrac{1}{2}$ transforming with respect to
the dual Weil representation, having the same principal part as $\widetilde{H}_{(6),1}$. In the light
of Lemma \ref{XiSpitze}, this implies that $\widetilde{H}_{(6),1}-\mathcal{I}^M_{1,1}(e_{(6),1},\tau)$ is a
cusp form and thus $\widetilde{H}_{(6),1}=\mathcal{I}^M_{1,1}(e_{(6),1},\tau)$ as
$S_{1/2,\overline{\rho}_L}=\{0\}$. Using the result of Theorem \ref{Alfes}, the holomorphic
coefficients of $\mathcal{I}^M_{1,1}(e_{(6),1},\tau)$ at
$q^{(48n-r^2)/48}\mathfrak{e}_r$ for $r^2-48n<0$ are given by
$$\frac{i}{\sqrt{48n-r^2}}\big(\textup{tr}^+_{e_{(6),1}}(r^2-48n,r)-\textup{tr}^-_{e_{(6),1}}(r^2-48n,r)\big).
$$ Comparing the coefficients of the holomorphic parts of both $\widetilde{H}_{(6),1}$ and
$\mathcal{I}^M_{1,1}(e_{(6),1},\tau)$ yields the stated formulas.
\end{proof}
\begin{thm}
Let $e_{(6),1}(z)\in M_0^!(12)$ and $e_{(6),2}(z)\in M_0^!(12)$ be defined as in
(\ref{e61}) and (\ref{e62}) and
put $E_{(6)}(z):=e_{(6),2}(z)+3e_{(6),1}(z)$.  
\begin{enumerate}
\item[(1)] For $n\geq0$ the coefficients $a_{2\mu}(n)$ of $2\mu(q)$ are given by
\begin{align*}
a_{2\mu}(n)&=\frac{i}{2\sqrt{48n-4}}\big(\textup{tr}^+_{E_{(6)}}(4-48n,2)-\textup{tr}^-_{E_{(6)}}(4-48n,2)\big).
\end{align*}
\item[(2)] For $n\geq0$ the coefficients $a_\lambda(n)$ of $\lambda(q)$ are given by
\begin{align*}
a_\lambda(n)&=\frac{i}{4\sqrt{48n-36}}\big(\textup{tr}^+_{E_{(6)}}(36-48n,6)-\textup{tr}^-_{E_{(6)}}(36-48n,6)\big).
\end{align*}
\item[(3)] For $n\geq0$ the coefficients $a_\nu(n)$ of $\nu(q)$ are given by
\begin{align*}
a_\nu(n)&= \begin{cases}
-\frac{i}{8\sqrt{48n-1}}\big(\textup{tr}^+_{E_{(6)}}(1-48n,1)-\textup{tr}^-_{E_{(6)}}(1-48n,1)\big),&\text{
if }n\text{ is even,} \\
-\frac{i}{8\sqrt{48n-25}}\big(\textup{tr}^+_{E_{(6)}}(25-48n,5)-\textup{tr}^-_{E_{(6)}}(25-48n,5)\big),&\text{
if }n\text{ is odd.} \end{cases}
\end{align*}
\item[(4)] For $n\geq0$ the coefficients $a_\xi(n)$ of $\xi(q)$ are given by
\begin{align*}
a_\xi(n)&=\begin{cases}
-\frac{i}{16\sqrt{48n-9}}\big(\textup{tr}^+_{E_{(6)}}(9-48n,3)-\textup{tr}^-_{E_{(6)}}(9-48n,3)\big),&\text{
if }n\text{ is even,} \\
\frac{i}{16\sqrt{48(n+1)-81}}\big(\textup{tr}^+_{E_{(6)}}(81-48(n+1),9)-\textup{tr}^-_{E_{(6)}}(81-48(n+1),9)\big),&\text{
if }n\text{ is odd.} \end{cases}
\end{align*}
\end{enumerate}
\end{thm}
\begin{proof}
This can be proven analogously to Theorem \ref{Formeln61}.
\end{proof}
\begin{rem}
The stated formulas were checked numerically using Sage \cite{sagemath}. 
\end{rem}

\subsection{Mock Theta Functions of order $2$} \label{section mock theta functions of order 2}

In this subsection we consider the mock theta functions $A$, $B$ and $\mu$ of order $2$ and prove similar results for their completions to harmonic weak Maass forms as in subsection \ref{section mock theta functions of order 6}. We omit the proofs here since all results of this subsection can be proven analogously to the previous subsection.

\begin{defin}
For $\tau \in \mathbb{H}$ we define the vector valued functions
$$F_{(2)}(\tau) := \left(\begin{matrix}
  4 \ q^{-\frac{1}{16}} \ A(q^{\frac{1}{2}}) \\
  4 \ q^{-\frac{1}{16}} \ A(-q^{\frac{1}{2}}) \\
  \sqrt{8} \ q^{\frac{1}{4}} \ B(q^{\frac{1}{2}}) \\
  \sqrt{8} \ q^{\frac{1}{4}} \ B(-q^{\frac{1}{2}}) \\
  q^{-\frac{1}{16}} \ \mu(q^{\frac{1}{2}}) \\
  q^{-\frac{1}{16}} \ \mu(-q^{\frac{1}{2}}) 
\end{matrix}\right),$$

where $q=e^{2 \pi i \tau}$, and
$$G_{(2)}(\tau) := \frac{i}{\sqrt{2}} \ \int_{-\overline{\tau}}^{i \infty} \frac{g_{(2)}(z)}{\sqrt{-i(z+\tau)}} \ dz,$$

where $g_{(2)}$ is the vector $(g_{(2),0},\dots,g_{(2),5})^{T}$ with components
\begin{flalign*}
g_{(2),0}(z) &:= \theta_{4,1}(z)+\theta_{4,3}(z), \\
g_{(2),1}(z) &:= \theta_{4,1}(z)-\theta_{4,3}(z), \\
g_{(2),2}(z) &:= \sqrt{2} \ \theta_{4,2}(z), \\
g_{(2),3}(z) &:= -\sqrt{2} \ \theta_{4,2}(z), \\
g_{(2),4}(z) &:= -(\theta_{4,1}(z)-\theta_{4,3}(z)), \\
g_{(2),5}(z) &:= -(\theta_{4,1}(z)+\theta_{4,3}(z)).
\end{flalign*}
\end{defin}

The so-defined functions $F_{(2)}$ and $G_{(2)}$ have the same modular transformation properties. As before we can consider $F_{(2)}-G_{(2)}$ which will be a vector valued harmonic weak Maass form as the following theorem states:

\begin{thm}
The function $H_{(2)}$, defined for $\tau \in \mathbb{H}$ by
$$H_{(2)}(\tau):=F_{(2)}(\tau)-G_{(2)}(\tau),$$

is a vector valued harmonic weak Maass form of weight $\nicefrac{1}{2}$ for the metaplectic group $\emph{\text{Mp}}_{2}(\mathbb{Z})$. \\
For $\tau \in \mathbb{H}$ we have
\begin{equation}
H_{(2)}(\tau+1) = \left(\begin{matrix}
  0 & \zeta_{16}^{-1} & 0 & 0 & 0 & 0 \\
  \zeta_{16}^{-1} & 0 & 0 & 0 & 0 & 0 \\
  0 & 0 & 0 & i & 0 & 0 \\
  0 & 0 & i & 0 & 0 & 0 \\
  0 & 0 & 0 & 0 & 0 & \zeta_{16}^{-1} \\
  0 & 0 & 0 & 0 & \zeta_{16}^{-1} & 0
\end{matrix}\right) \ H_{(2)}(\tau)
\end{equation}

and
\begin{equation}
 H_{(2)} \left(-\frac{1}{\tau} \right) = \sqrt{-i \tau} \ \left(\begin{matrix}
  0 & 0 & 0 & 0 & 0 & 1 \\
  0 & 0 & 0 & 1 & 0 & 0 \\
  0 & 0 & 0 & 0 & 1 & 0 \\
  0 & 1 & 0 & 0 & 0 & 0 \\
  0 & 0 & 1 & 0 & 0 & 0 \\
  1 & 0 & 0 & 0 & 0 & 0
\end{matrix}\right) \ H_{(2)}(\tau).
\end{equation}

\end{thm}

\begin{cor} \label{corollary xi-image order 2}
We have $\xi_{1/2}(H_{(2)})(\tau)=-g_{(2)}(\tau)$.
\end{cor}

After we have constructed a vector valued harmonic weak Maass form that contains mock theta functions of order $2$, we again take a closer look at its components. We consider
$$\Gamma(2) = \left\{\left(\begin{matrix}
  a & b \\
  c & d \\
\end{matrix}\right) \in \text{SL}_{2}(\mathbb{Z}) \ \Big| \ b \equiv c \equiv 0 \ (2), \ a \equiv d \equiv 1 \ (2) \right\},$$

the principal congruence subgroup of level $2$, and obtain the following result:

\begin{thm} \label{theorem order 2 scalar valued harmonic Maass forms}
The components of the vector valued harmonic weak Maass form $H_{(2)}$ are scalar valued harmonic weak Maass forms of weight $\nicefrac{1}{2}$ for the subgroup
$$\{(\gamma,\phi) \in \emph{\text{Mp}}_{2}(\mathbb{Z}) \ | \ \gamma \in \Gamma(2) \}$$

of the metaplectic group $\emph{\text{Mp}}_{2}(\mathbb{Z})$.
\end{thm}

So we have interpreted all second order mock theta functions as the holomorphic part of a scalar valued harmonic weak Maass form.

\begin{rem}
As in the previous section, the $\xi$-images of the harmonic weak Maass forms in Theorem \ref{theorem order 2 scalar valued harmonic Maass forms} follow immediately from Corollary \ref{corollary xi-image order 2}.
\end{rem}

The shadows of the second order mock theta functions were also computed by Kang and Swisher in \cite{kang2017mock}.

\subsection{Mock Theta Functions of order $3$} \label{section mock theta functions of order 3}

We now turn to the mock theta functions $\phi$, $\psi$ and $\nu$ of order $3$. As before, we omit proofs in this subsection.

\begin{defin}
For $\tau \in \mathbb{H}$ we define the vector valued functions
$$F_{(3)}(\tau) := \left(\begin{matrix}
  q^{-\frac{1}{48}} \ \phi(q^{\frac{1}{2}}) \\
  q^{-\frac{1}{48}} \ \phi(-q^{\frac{1}{2}}) \\
  2 \ q^{-\frac{1}{48}} \ \psi(q^{\frac{1}{2}}) \\
  2 \ q^{-\frac{1}{48}} \ \psi(-q^{\frac{1}{2}}) \\
  \sqrt{2} \ q^{\frac{1}{6}} \ \nu(q^{\frac{1}{2}}) \\
  \sqrt{2} \ q^{\frac{1}{6}} \ \nu(-q^{\frac{1}{2}}) 
\end{matrix}\right),$$

where $q=e^{2 \pi i \tau}$, and
$$G_{(3)}(\tau) := \frac{i}{\sqrt{24}} \ \int_{-\overline{\tau}}^{i \infty} \frac{g_{(3)}(z)}{\sqrt{-i(z+\tau)}} \ dz,$$

where $g_{(3)}$ is the vector $(g_{(3),0},\dots,g_{(3),5})^{T}$ with components
\begin{flalign*}
g_{(3),0}(z) &:= -(\theta_{12,1}(z)+\theta_{12,5}(z)+\theta_{12,7}(z)+\theta_{12,11}(z)), \\
g_{(3),1}(z) &:= -(\theta_{12,1}(z)-\theta_{12,5}(z)+\theta_{12,7}(z)-\theta_{12,11}(z)), \\
g_{(3),2}(z) &:= \theta_{12,1}(z)+\theta_{12,5}(z)+\theta_{12,7}(z)+\theta_{12,11}(z), \\
g_{(3),3}(z) &:= \theta_{12,1}(z)-\theta_{12,5}(z)+\theta_{12,7}(z)-\theta_{12,11}(z), \\
g_{(3),4}(z) &:= -\sqrt{2} \ (\theta_{12,4}(z)+\theta_{12,8}(z)), \\
g_{(3),5}(z) &:= \sqrt{2} \ (\theta_{12,4}(z)+\theta_{12,8}(z)).
\end{flalign*}
\end{defin}

Since these two functions have the same modular tranformation properties we find for the function $F_{(3)}-G_{(3)}$:

\begin{thm} \label{theorem transformation properties of H_(3)}
The function $H_{(3)}$, defined for $\tau \in \mathbb{H}$ by
$$H_{(3)}(\tau):=F_{(3)}(\tau)-G_{(3)}(\tau),$$

is a vector valued harmonic weak Maass form of weight $\nicefrac{1}{2}$ for the metaplectic group $\emph{\text{Mp}}_{2}(\mathbb{Z})$. \\
For $\tau \in \mathbb{H}$ we have
\begin{equation}
H_{(3)}(\tau+1) = \left(\begin{matrix}
  0 & \zeta_{48}^{-1} & 0 & 0 & 0 & 0 \\
  \zeta_{48}^{-1} & 0 & 0 & 0 & 0 & 0 \\
  0 & 0 & 0 & \zeta_{48}^{-1} & 0 & 0 \\
  0 & 0 & \zeta_{48}^{-1} & 0 & 0 & 0 \\
  0 & 0 & 0 & 0 & 0 & \zeta_{6} \\
  0 & 0 & 0 & 0 & \zeta_{6} & 0
\end{matrix}\right) \ H_{(3)}(\tau)
\end{equation}

and
\begin{equation}
 H_{(3)} \left(-\frac{1}{\tau} \right) = \sqrt{-i \tau} \ \left(\begin{matrix}
  0 & 0 & 1 & 0 & 0 & 0 \\
  0 & 0 & 0 & 0 & 0 & 1 \\
  1 & 0 & 0 & 0 & 0 & 0 \\
  0 & 0 & 0 & 0 & 1 & 0 \\
  0 & 0 & 0 & 1 & 0 & 0 \\
  0 & 1 & 0 & 0 & 0 & 0
\end{matrix}\right) \ H_{(3)}(\tau).
\end{equation}

\end{thm}

\begin{cor} \label{corollary xi-image order 3}
We have $\xi_{1/2}(H_{(3)})(\tau)=-\frac{1}{\sqrt{12}} \ g_{(3)}(\tau)$.
\end{cor}

We now want to complete the mock theta functions $\phi,\psi$ and $\nu$ to scalar valued harmonic weak Maass forms. We again consider the group $\Gamma(2)$ and obtain:

\begin{thm} \label{theorem order 3 scalar valued harmonic Maass forms}
The components of the vector valued harmonic weak Maass form $H_{(3)}$ are scalar valued harmonic weak Maass forms of weight $\nicefrac{1}{2}$ for the subgroup
$$\{(\gamma,\phi) \in \emph{\text{Mp}}_{2}(\mathbb{Z}) \ | \ \gamma \in \Gamma(2) \}$$

of the metaplectic group $\emph{\text{Mp}}_{2}(\mathbb{Z})$.
\end{thm}

Thus we have related the mock theta functions $\phi,\psi$ and $\nu$ to scalar valued harmonic weak Maass forms.

\begin{rem}
Again we get the $\xi$-images of the harmonic weak Maass forms in Theorem \ref{theorem order 3 scalar valued harmonic Maass forms} from Corollary \ref{corollary xi-image order 3}.
\end{rem}
The mock theta functions $f$ and $\omega$ of order $3$  have already been treated by Zwegers in \cite{zwegers2001mock}, and Bruinier and Schwagenscheidt in \cite{bruinier2017algebraic} and we state their results for completeness.
\begin{thm}[\cite{zwegers2001mock}, Theorem 3.6] 
The vector $$F_3(\tau)=\begin{pmatrix}
q^{-\frac{1}{24}}f(q)\\2\ q^{\frac{1}{3}}\ \omega(q^{\frac{1}{2}})\\2 \ q^{\frac{1}{3}}\ \omega(-q^{\frac{1}{2}})
\end{pmatrix} $$ is the holomorphic part of a harmonic weak Maass form $H_{3}=(h_0,h_1,h_2)^T\in H^+_{1/2}$ of weight $\nicefrac{1}{2}$, transforming as 
\begin{align*}
H_3(\tau+1)=\begin{pmatrix}
\zeta_{24}^{-1}&0&0\\0&0&\zeta_3\\0&\zeta_3&0
\end{pmatrix} \ H_3(\tau)
\end{align*}
and
\begin{align*}
H_3\bigg(-\frac{1}{\tau}\bigg)=\sqrt{-i \tau} \ \begin{pmatrix}
0&1&0\\1&0&0\\0&0&-1
\end{pmatrix} \ H_3(\tau).
\end{align*}
\end{thm} 
This result can be used to construct a harmonic weak Maass form that transforms with respect to the dual Weil representation.
\begin{lemma}
The function $$\widetilde{H}_3=h_0 \ [\mathfrak{e}_1-\mathfrak{e}_5+\mathfrak{e}_7-\mathfrak{e}_{11}]+(h_2-h_1) \ [\mathfrak{e}_2-\mathfrak{e}_{10}]+(h_1+h_2) \ [-\mathfrak{e}_4+\mathfrak{e}_8] $$ transforms like a vector valued modular form of weight $\nicefrac{1}{2}$ with respect to the dual Weil representation $\overline{\rho}_L$ of level $N=6$.
\end{lemma}
Let $E_4$ denote the normalized Eisenstein series of weight $4$ for $\text{SL}_2(\mathbb{Z})$. We consider the function \begin{align}\label{e3}
e_{(3)}(z):=-\frac{1}{40}\ \frac{E_4(z)+4E_4(2z)-9E_4(3z)-36E_4(6z)}{(\eta(z)\eta(2z)\eta(3z)\eta(6z))^2}
\end{align} which is a weakly holomorphic modular form of weight $0$, level $6$ and whose principal part starts with $q^{-1}$.
\begin{thm}[\cite{bruinier2017algebraic}, Theorem 3.1]
Let $e_{(3)}\in M_0^!(6)$ be the function defined in (\ref{e3}).
\begin{itemize}
\item[(1)] For $n\geq1$ the coefficients $a_f(n)$ of $f(q)$ are given by $$a_f(q)=\frac{i}{2\sqrt{24n-1}}\big(\textup{tr}_{e_{(3)}}^+(1-24n,1)-\textup{tr}_{e_{(3)}}^-(1-24n,1)\big). $$
\item[(2)] For $n\geq1$ the coefficients $a_\omega(n)$ of $\omega(q)$ are given by
\begin{align*}
a_\omega(q)=\begin{cases}\frac{-i}{8\sqrt{24(\frac{n}{2}+1)-16}}\big(\textup{tr}_{e_{(3)}}^+(16-24(\frac{n}{2}+1),4)-\textup{tr}_{e_{(3)}}^-(16-24(\frac{n}{2}+1),4)\big), &\text{ if } n \text{ is even,}\\\frac{-i}{8\sqrt{24\frac{n+1}{2}-4}}\big(\textup{tr}_{e_{(3)}}^+(4-24\frac{n+1}{2},2)-\textup{tr}_{e_{(3)}}^-(4-24\frac{n+1}{2},2)\big), &\text{ if } n \text{ is odd.}  \end{cases}
\end{align*}
\end{itemize}
\end{thm}

\subsection{Mock Theta Functions of order $5$} \label{section mock theta functions of
order 5}
For the mock theta functions of order 5 the necessary completions and their
transformation properties have already been studied by Zwegers and Andersen in
\cite{zwegers2008mock} and \cite{andersen2016}, respectively. Using their results we derive algebraic formulas for their coefficients. The proofs are analogous to the corresponding proofs in subsection \ref{section mock theta functions of order 6}. \\
We define the two matrices
\begin{align}\label{Matrix51}
N_{(5)}=\begin{pmatrix}
\zeta_{60}^{-1}&0&0&0&0&0\\0&\zeta_{60}^{11}&0&0&0&0\\0&0&0&0&\zeta_{240}^{-1}&0\\0&0&0&0&0&\zeta_{240}^{71}\\0&0&\zeta_{240}^{-1}&0&0&0\\0&0&0&\zeta_{240}^{71}&0&0
\end{pmatrix}
\end{align}
and
\begin{align}\label{Matrix52}
M_{(5)}=\begin{pmatrix}
0&0&\sqrt{2} \ \sin(\frac{\pi}{5})&\sqrt{2} \ \sin(\frac{2\pi}{5})&0&0\\0&0&\sqrt{2} \ \sin(\frac{2\pi}{5})&-\sqrt{2} \ \sin(\frac{\pi}{5})&0&0\\
\frac{1}{\sqrt{2}} \ \sin(\frac{\pi}{5})&\frac{1}{\sqrt{2}} \ \sin(\frac{2\pi}{5})&0&0&0&0\\\frac{1}{\sqrt{2}} \ \sin(\frac{2\pi}{5})&-\frac{1}{\sqrt{2}} \ \sin(\frac{\pi}{5})&0&0&0&0\\0&0&0&0&\sin(\frac{2\pi}{5})&\sin(\frac{\pi}{5})\\0&0&0&0&\sin(\frac{\pi}{5})&\sin(\frac{2\pi}{5})
\end{pmatrix}.
\end{align}
\begin{thm}[\cite{zwegers2008mock}, Proposition 4.10]\label{WeilTrafo51}
The vector
$$F_{(5),1}(\tau)=\begin{pmatrix}
q^{-\frac{1}{60}} \ f_0(q)\\q^{\frac{11}{60}} \ f_1(q)\\q^{-\frac{1}{240}}\ \big(-1+F_0(q^{1/2})\big)\\q^{\frac{71}{240}} \ F_1(q^{1/2})\\q^{-\frac{1}{240}}\ (-1+F_0\big(-q^{1/2})\big)\\q^{\frac{71}{240}} \ F_1(-q^{1/2})
\end{pmatrix} $$
is the holomorphic part of
$H_{(5),1}=(f_{4,1},f_{196,1},f_{1,1},f_{169,1},g_{1,1},g_{169,1})^T\in H_{1/2}^+$, which is a harmonic weak Maass form of weight $\nicefrac{1}{2}$, transforming as
\begin{align}
H_{(5),1}(\tau+1)=N_{(5)} \ H_{(5),1}(\tau)
\end{align}
and
\begin{align}
H_{(5),1}\bigg(-\frac{1}{\tau}\bigg)=\sqrt{-i\tau} \ \frac{2}{\sqrt{5}} \ M_{(5)} \ H_{(5),1}(\tau),
\end{align}
where the matrices $N_{(5)}$ and $M_{(5)}$ are defined as in (\ref{Matrix51}) and
(\ref{Matrix52}).
\end{thm}
\begin{thm}[\cite{zwegers2008mock}, Proposition 4.13]\label{WeilTrafo52}
The vector
$$F_{(5),2}(\tau)=\begin{pmatrix}
2 \ q^{-\frac{1}{60}} \ \psi_0(q)\\2 \ q^{\frac{11}{60}} \ \psi_1(q)\\q^{-\frac{1}{240}} \ \varphi_0(-q^{\frac{1}{2}})\\-q^{-\frac{49}{240}} \ \varphi_1(-q^{\frac{1}{2}})\\q^{-\frac{1}{240}} \ \varphi_0(q^{\frac{1}{2}})\\q^{-\frac{49}{240}} \ \varphi_1(q^{\frac{1}{2}})
\end{pmatrix} $$
is the holomorphic part of
$H_{(5),2}=(f_{4,2},f_{196,2},f_{1,2},f_{169,2},g_{1,2},g_{169,2})^T\in H_{1/2}^+$, which is a harmonic weak Maass form of weight $\nicefrac{1}{2}$, transforming as
\begin{align}
H_{(5),2}(\tau+1)=N_{(5)} \ H_{(5),2}(\tau)
\end{align}
and
\begin{align}
H_{(5),2}\bigg(-\frac{1}{\tau}\bigg)=\sqrt{-i\tau} \ \frac{2}{\sqrt{5}} \ M_{(5)} \ H_{(5),2}(\tau),
\end{align}
where the matrices $N_{(5)}$ and $M_{(5)}$ are defined as in (\ref{Matrix51}) and
(\ref{Matrix52}).
\end{thm}
\begin{lemma}[\cite{andersen2016}, Lemma 5]\label{5VolleWeilDarst}
Suppose that $(f_{4,1},f_{196,1},f_{1,1},f_{169,1},g_{1,1},g_{169,1})^T$ transforms
with the representation given in Theorem \ref{WeilTrafo51}, and that
$(f_{4,2},f_{196,2},f_{1,2},f_{169,2},g_{1,2},g_{169,2})^T$ transforms with the
representation given in Theorem \ref{WeilTrafo52}. For $j=1,2$ we define the
function
\begin{align*}
\widetilde{H}_{(5),j}&=\sum\limits_{\substack{0<r<60\\r\equiv\pm1\ (10)\\\gcd(r,60)=1}}(a_r \ f_{1,j}+b_r \ g_{1,j}) \ (\mathfrak{e}_r-\mathfrak{e}_{-r})-\sum\limits_{\substack{0<r<60\\r\equiv\pm2\ (10)\\\gcd(r,60)=2}}f_{4,j} \ (\mathfrak{e}_r-\mathfrak{e}_{-r})\\&+\sum\limits_{\substack{0<r<60\\r\equiv\pm3\ (10)\\\gcd(r,60)=1}}(a_r \ f_{169,j}+b_r \ g_{169,j}) \ (\mathfrak{e}_r-\mathfrak{e}_{-r})-\sum\limits_{\substack{0<r<60\\r\equiv\pm4\ (10)\\\gcd(r,60)=2}}f_{196,j} \ (\mathfrak{e}_r-\mathfrak{e}_{-r}),
\end{align*}
where
$$a_r=\begin{cases}+1\quad\text{if }0<r<30,\\-1\quad\text{otherwise,}
\end{cases}\quad\text{and}\quad b_r=\begin{cases}+1\quad\text{if
}r\equiv\pm1,\pm13\ (60),\\-1\quad\text{otherwise}. \end{cases} $$
Then $\widetilde{H}_{(5),1}\in H_{1/2,\overline{\rho}_L}^+$ and $\widetilde{H}_{(5),2}\in H_{1/2,\overline{\rho}_L}^+$ both
transform like a vector valued modular form of weight $\nicefrac{1}{2}$ for the dual Weil
representation $\overline{\rho}_L$ of level $N=60$.
\end{lemma}

\begin{defin}
We define the functions
\begin{align}\label{e51}
e_{(5),1}(z):=\frac{\eta(z)\eta(12z)\eta(15z)\eta(20z)}{\eta(3z)\eta(4z)\eta(5z)\eta(60z)}-\frac{\eta(3z)\eta(4z)\eta(5z)\eta(60z)}{\eta(z)\eta(12z)\eta(15z)\eta(20z)}
\end{align}
and
\begin{align}\label{e52}
e_{(5),2}:=\bigg(
\frac{\eta(z)\eta(12z)\eta(15z)\eta(20z)}{\eta(3z)\eta(4z)\eta(5z)\eta(60z)}\bigg)^2-\bigg(\frac{\eta(3z)\eta(4z)\eta(5z)\eta(60z)}{\eta(z)\eta(12z)\eta(15z)\eta(20z)}\bigg)^2.
\end{align}
\end{defin}
These functions are weakly holomorphic modular forms of weight $0$, level $60$ whose
principal parts start with $q^{-1}$ and $q^{-2}$, respectively.
\begin{thm}
Let $e_{(5),1}(z),e_{(5),2}(z)\in M_0^{!}(60)$ be defined as in (\ref{e51}) and
(\ref{e52}) and put
$E_{(5)}(z):=-e_{(5),2}-e_{(5),1}$.
\begin{itemize}
\item[(1)] For $n\geq1$ the coefficients $a_{f_0}(n)$ of $f_0(q)$ are given by
$$a_{f_0}(n)=\frac{-i}{2\sqrt{240n-4}}\big(\textup{tr}_{E_{(5)}}^+(4-240n,2)-\textup{tr}_{E_{(5)}}^-(4-240n,2)\big).
$$
\item[(2)] For $n\geq1$ the coefficients $a_{f_1}(n)$ of $f_1(q)$ are given by
$$a_{f_1}(n)=\frac{-i}{2\sqrt{240(n+1)-196}}\big(\textup{tr}_{E_{(5)}}^+(196-240(n+1),14)-\textup{tr}_{E_{(5)}}^-(196-240(n+1),14)\big).
$$
\item[(3)] For $n\geq1$ the coefficients $a_{F_0}(n)$ of $F_0(q)$ are given by
\begin{align*}
a_{F_0}(n)=\begin{cases}
\frac{i}{4\sqrt{240\frac{n}{2}-1}}\big(\textup{tr}_{E_{(5)}}^+(1-240\frac{n}{2},1)-\textup{tr}_{E_{(5)}}^-(1-240\frac{n}{2},1)\big),
&\text{if } n\ \text{is even},
\\\frac{i}{4\sqrt{240\frac{n+1}{2}-121}}\big(\textup{tr}_{E_{(5)}}^+(121-240\frac{n+1}{2},11)-\textup{tr}_{E_{(5)}}^-(121-240\frac{n+1}{2},11)\big),
&\text{if } n\ \text{is odd.} \end{cases} 
\end{align*}
\item[(4)] For $n\geq1$ the coefficients $a_{F_1}(n)$ of $F_1(q)$ are given by
\begin{align*}
a_{F_1}(n)=\begin{cases}
\frac{i}{4\sqrt{240\frac{n+2}{2}-169}}\big(\textup{tr}_{E_{(5)}}^+(169-240\frac{n+2}{2},13)-\textup{tr}_{E_{(5)}}^-(169-240\frac{n+2}{2}),13)\big),&\text{if
}n\
\text{is even},\\\frac{i}{4\sqrt{240\frac{n+1}{2}-49}}\big(\textup{tr}_{E_{(5)}}^+(49-240\frac{n+1}{2},7)-\textup{tr}_{E_{(5)}}^-(49-240\frac{n+1}{2},7)\big),&\text{if
}n\ \text{is odd.} \end{cases}
\end{align*}
\end{itemize}
\end{thm}

\begin{thm}
Let $e_{(5),1}\in M_0^{!}(60)$ be defined as in (\ref{e51}).
\begin{itemize}
\item[(1)] For $n\geq1$ the coefficients $a_{\psi_0}(n)$ of $\psi_0(q)$ are given by
$$a_{\psi_0}(n)=\frac{-i}{2\sqrt{240n-4}}\big(\textup{tr}_{e_{(5),1}}^+(4-240n,2)-\textup{tr}_{e_{(5),1}}^-(4-240n,2)\big).
$$
\item[(2)] For $n\geq1$ the coefficients $a_{\psi_1}(n)$ of $\psi_1(q)$ are given by
$$a_{\psi_1}(n)=\frac{-i}{2\sqrt{240(n+1)-196}}\big(\textup{tr}_{e_{(5),1}}^+(196-240(n+1),14)-\textup{tr}_{e_{(5),1}}^-(196-240(n+1),14)\big).
$$
\item[(3)] For $n\geq1$ the coefficients $a_{\varphi_0}(n)$ of $\varphi_0(q)$ are
given by
\begin{align*}
a_{\varphi_0}(n)=\begin{cases}\frac{i}{2\sqrt{240\frac{n}{2}-1}}\big(\textup{tr}_{e_{(5),1}}^+(1-240\frac{n}{2},1)-\textup{tr}_{e_{(5),1}}^-(1-240\frac{n}{2},1)\big),
&\text{if } n\ \text{is even},
\\\frac{-i}{2\sqrt{240\frac{n+1}{2}-121}}\big(\textup{tr}_{e_{(5),1}}^+(121-240\frac{n+1}{2},11)-\textup{tr}_{e_{(5),1}}^-(121-240\frac{n+1}{2},11)\big),
&\text{if } n\ \text{is odd.} \end{cases}
\end{align*}
\item[(4)] For $n\geq1$ the coefficients $a_{\varphi_1}(n)$ of $\varphi_1(q)$ are
given by
\begin{align*}
a_{\varphi_1}(n)=\begin{cases}\frac{-i}{2\sqrt{240\frac{n}{2}-49}}\big(\textup{tr}_{e_{(5),1}}^+(49-240\frac{n}{2},7)-\textup{tr}_{e_{(5),1}}^-(49-240\frac{n}{2},7)\big),
&\text{if }n\
\text{is even},\\\frac{i}{2\sqrt{240\frac{n+1}{2}-169}}\big(\textup{tr}_{e_{(5),1}}^+(169-240\frac{n+1}{2},13)-\textup{tr}_{e_{(5),1}}^-(169-240\frac{n+1}{2},13)\big),
&\text{if }n\ \text{is odd.} \end{cases}
\end{align*}
\end{itemize}
\end{thm}

\subsection{Mock Theta Functions of order $7$} \label{section mock theta functions of
order 7}
Similar to the previous subsection the necessary completion and its transformation behaviour have already been studied by Zwegers and Andersen in \cite{zwegers2008mock} and \cite{andersen2016vector}, respectively. We use their results to derive algebraic formulas for the coefficients of the seventh order mock theta functions.
\begin{thm}[\cite{zwegers2008mock}, Proposition 4.5]
The vector 
\begin{align*}
F_{(7)}(\tau)=\begin{pmatrix}
q^{-\frac{1}{168}} \ \mathcal{F}_0(q)\\q^{\frac{47}{168}} \ \mathcal{F}_2(q)\\q^{-\frac{25}{168}} \ \mathcal{F}_1(q)
\end{pmatrix}
\end{align*}
is the holomorphic part of a harmonic weak Maass form $H_{(7)}=(f_1,f_{121},f_{25})^T\in
H_{1/2}^+$ of weight $\nicefrac{1}{2}$, transforming as
\begin{align}
H_{(7)}(\tau+1)=\begin{pmatrix}
\zeta_{168}^{-1}&0&0\\0&\zeta_{168}^{47}&0\\0&0&\zeta_{168}^{-25}
\end{pmatrix} \ H_{(7)}(\tau)
\end{align}
and
\begin{align}
H_{(7)}\bigg(-\frac{1}{\tau}\bigg)=\sqrt{-i\tau} \ \frac{2}{\sqrt{7}} \ \begin{pmatrix}
\sin(\frac{\pi}{7})&\sin(\frac{3\pi}{7})&\sin(\frac{2\pi}{7})\\\sin(\frac{3\pi}{7})&-\sin(\frac{2\pi}{7})&\sin(\frac{\pi}{7})\\\sin(\frac{2\pi}{7})&\sin(\frac{\pi}{7})&-\sin(\frac{3\pi}{7})
\end{pmatrix} \ H_{(7)}(\tau).
\end{align}
\label{zwegerstrafo}
\end{thm}
\begin{lemma}[\cite{andersen2016vector}, Lemma 4] \label{7VolleWeilDarst}
Suppose that $(f_1,f_{121},f_{25})^T$ transforms with the representation given in
Theorem \ref{zwegerstrafo}. Then the function
\begin{align*}
\widetilde{H}_{(7)}=\sum\limits_{r\in \mathbb{Z}/168\mathbb{Z}}\widetilde{H}_r\mathfrak{e}_r=
f_1 \ (\mathfrak{e}_1-\mathfrak{e}_{-1})+f_1 \ (\mathfrak{e}_{41}-\mathfrak{e}_{-41})-\sum\limits_{\substack{2\leq
r\leq 40\\r^2 \ (168)\in\{1,25,121\}}}f_{r^2} \ (\mathfrak{e}_r-\mathfrak{e}_{-r})
\end{align*}
transforms like a vector valued modular form of weight $\nicefrac{1}{2}$ for the dual Weil representation $\overline{\rho}_L$ of level $N=42$, so that $\widetilde{H}_{(7)}\in
H_{1/2,\overline{\rho}_L}^+$.
\end{lemma}
\begin{defin}
We define the function
\begin{align}\label{e7}
e_{(7)}(z):=\Bigg(\frac{\eta(z)\eta(6z)\eta(14z)\eta(21z)}{\eta(2z)\eta(3z)\eta(7z)\eta(42z)}\Bigg)^{2}-\Bigg(\frac{\eta(2z)\eta(3z)\eta(7z)\eta(42z)}{\eta(z)\eta(6z)\eta(14z)\eta(21z)}\Bigg)^{2}.
\end{align}
\end{defin}
This function is a weakly holomorphic modular form of level $42$, weight $0$ whose
principal part starts with $q^{-1}$.
\begin{thm}
Let $e_{(7)}\in M_0^{!}(42)$ be defined as in (\ref{e7}).
\begin{enumerate}
\item[(1)] For $n\geq1$ the coefficients $a_{\mathcal{F}_0}(n)$ of
$\mathcal{F}_0(q)$ are given by
\begin{align*}
a_{\mathcal{F}_0}(n)=\frac{i}{2\sqrt{168n-1}}\big(\textup{tr}^+_{e_{(7)}}(1-168n,1)-\textup{tr}^-_{e_{(7)}}(1-168n,1)\big).
\end{align*}
\item[(2)] For $n\geq1$ the coefficients $a_{\mathcal{F}_1}(n)$ of
$\mathcal{F}_1(q)$ are given by
\begin{align*}
a_{\mathcal{F}_1}(n)=\frac{-i}{2\sqrt{168n-25}}\big(\textup{tr}^+_{e_{(7)}}(25-168n,5)-\textup{tr}^-_{e_{(7)}}(25-168n,5)\big).
\end{align*}
\item[(3)] For $n\geq1$ the coefficients $a_{\mathcal{F}_2}(n)$ of
$\mathcal{F}_2(q)$ are given by
\begin{align*}
a_{\mathcal{F}_2}(n)=\frac{-i}{2\sqrt{168(n+1)-121}}\big(\textup{tr}^+_{e_{(7)}}(121-168(n+1),11)-\textup{tr}^-_{e_{(7)}}(121-168(n+1),11)\big).
\end{align*}
\end{enumerate}
\end{thm}

\subsection{Mock Theta Functions of order $8$} \label{section mock theta functions of order 8}

\setcounter{MaxMatrixCols}{20}

We now turn to the mock theta functions $S_{0}$, $S_{1}$, $T_{0}$, $T_{1}$, $U_{0}$, $U_{1}$, $V_{0}$ and $V_{1}$ of order $8$. They have the following linear relations between them which are an easy consequence of the identities that are, e.g., given as (1.7) and (1.8) in \cite{gordon2000some}.

\begin{lemma} \label{lemma connection between eighth order U_0, S_0 and S_1 and U_1, T_0 and T_1}

We have
\begin{flalign*}
q^{-\frac{1}{32}} \ U_{0}(q^{\frac{1}{4}}) &= q^{-\frac{1}{32}} \ S_{0}(q^{\frac{1}{2}}) + q^{\frac{7}{32}} \ S_{1}(q^{\frac{1}{2}}), \\
q^{-\frac{1}{32}} \ U_{0}(-q^{\frac{1}{4}}) &= q^{-\frac{1}{32}} \ S_{0}(q^{\frac{1}{2}}) - q^{\frac{7}{32}} \ S_{1}(q^{\frac{1}{2}}), \\
q^{-\frac{1}{32}} \ U_{1}(q^{\frac{1}{4}}) &= q^{-\frac{1}{32}} \ T_{0}(q^{\frac{1}{2}}) + q^{\frac{7}{32}} \ T_{1}(q^{\frac{1}{2}}), \\
q^{-\frac{1}{32}} \ U_{1}(-q^{\frac{1}{4}}) &= q^{-\frac{1}{32}} \ T_{0}(q^{\frac{1}{2}}) - q^{\frac{7}{32}} \ T_{1}(q^{\frac{1}{2}}).
\end{flalign*}

\end{lemma}

\begin{defin}
For $\tau \in \mathbb{H}$ we define the vector valued functions
$$F_{(8)}(\tau) := \left(\begin{matrix}
  V_{0}(q^{\frac{1}{2}}) \\
  V_{0}(-q^{\frac{1}{2}}) \\
  \sqrt{8} \ q^{-\frac{1}{8}} \ V_{1}(q^{\frac{1}{2}}) \\
  \sqrt{8} \ q^{-\frac{1}{8}} \ V_{1}(-q^{\frac{1}{2}}) \\
  \sqrt{2} \ q^{-\frac{1}{32}} \ S_{0}(q^{\frac{1}{2}}) \\
  \sqrt{2} \ q^{-\frac{1}{32}} \ S_{0}(-q^{\frac{1}{2}}) \\
  \sqrt{2} \ q^{\frac{7}{32}} \ S_{1}(q^{\frac{1}{2}}) \\
  \sqrt{2} \ q^{\frac{7}{32}} \ S_{1}(-q^{\frac{1}{2}}) \\
  \sqrt{8} \ q^{-\frac{1}{32}} \ T_{0}(q^{\frac{1}{2}}) \\
  \sqrt{8} \ q^{-\frac{1}{32}} \ T_{0}(-q^{\frac{1}{2}}) \\
  \sqrt{8} \ q^{\frac{7}{32}} \ T_{1}(q^{\frac{1}{2}}) \\
  \sqrt{8} \ q^{\frac{7}{32}} \ T_{1}(-q^{\frac{1}{2}}) 
\end{matrix}\right),$$

where $q=e^{2 \pi i \tau}$, and
$$G_{(8)}(\tau) := \frac{i}{\sqrt{8}} \ \int_{-\overline{\tau}}^{i \infty} \frac{g_{(8)}(z)}{\sqrt{-i(z+\tau)}} \ dz,$$

where $g_{(8)}$ is the vector $(g_{(8),0},\dots,g_{(8),11})^{T}$ with components
\begin{flalign*}
g_{(8),0}(z) &:= \sqrt{2} \ \theta_{8,4}(z), \\
g_{(8),1}(z) &:= -\sqrt{2} \ \theta_{8,4}(z), \\
g_{(8),2}(z) &:= \theta_{8,2}(z)+\theta_{8,6}(z), \\
g_{(8),3}(z) &:= \theta_{8,2}(z)+\theta_{8,6}(z), \\
g_{(8),4}(z) &:= -(\theta_{8,1}(z)-\theta_{8,7}(z)), \\
g_{(8),5}(z) &:= -(\theta_{8,1}(z)+\theta_{8,7}(z)), \\
g_{(8),6}(z) &:= \theta_{8,3}(z)-\theta_{8,5}(z), \\
g_{(8),7}(z) &:= -(\theta_{8,3}(z)+\theta_{8,5}(z)), \\
g_{(8),8}(z) &:= \theta_{8,1}(z)-\theta_{8,7}(z), \\
g_{(8),9}(z) &:= \theta_{8,1}(z)+\theta_{8,7}(z), \\
g_{(8),10}(z) &:= -(\theta_{8,3}(z)-\theta_{8,5}(z)), \\
g_{(8),11}(z) &:= \theta_{8,3}(z)+\theta_{8,5}(z).
\end{flalign*}
\end{defin}

Again the so defined functions have the same modular transformation properties. Considering the function $F_{(8)}-G_{(8)}$ leads to the following theorem:

\begin{thm} \label{theorem transformation properties of H_(8)}
The function $H_{(8)}$, defined for $\tau \in \mathbb{H}$ by
$$H_{(8)}(\tau):=F_{(8)}(\tau)-G_{(8)}(\tau) \ \ \ \ \ (\tau \in \mathbb{H})$$

is a vector valued harmonic Maass form of weight $\nicefrac{1}{2}$ for the metaplectic group $\emph{\text{Mp}}_{2}(\mathbb{Z})$. \\
For $\tau \in \mathbb{H}$ we have
\begin{equation}
H_{(8)}(\tau+1) = N_{(8)} \ H_{(8)}(\tau)
\end{equation}

and
\begin{equation}
 H_{(8)} \left(-\frac{1}{\tau} \right) = \sqrt{-i \tau} \ M_{(8)} \ H_{(8)}(\tau),
\end{equation}

where the transformation matrices $N_{(8)}$ and $M_{(8)}$ are defined as
$$N_{(8)} := \left(\begin{matrix}
  0 & 1 & 0 & 0 & 0 & 0 & 0 & 0 & 0 & 0 & 0 & 0 \\
  1 & 0 & 0 & 0 & 0 & 0 & 0 & 0 & 0 & 0 & 0 & 0 \\
  0 & 0 & 0 & \zeta_{8}^{-1} & 0 & 0 & 0 & 0 & 0 & 0 & 0 & 0 \\
  0 & 0 & \zeta_{8}^{-1} & 0 & 0 & 0 & 0 & 0 & 0 & 0 & 0 & 0 \\
  0 & 0 & 0 & 0 & 0 & \zeta_{32}^{-1} & 0 & 0 & 0 & 0 & 0 & 0 \\
  0 & 0 & 0 & 0 & \zeta_{32}^{-1} & 0 & 0 & 0 & 0 & 0 & 0 & 0 \\
  0 & 0 & 0 & 0 & 0 & 0 & 0 & \zeta_{32}^{7} & 0 & 0 & 0 & 0 \\
  0 & 0 & 0 & 0 & 0 & 0 & \zeta_{32}^{7} & 0 & 0 & 0 & 0 & 0 \\
  0 & 0 & 0 & 0 & 0 & 0 & 0 & 0 & 0 & \zeta_{32}^{-1} & 0 & 0 \\
  0 & 0 & 0 & 0 & 0 & 0 & 0 & 0 & \zeta_{32}^{-1} & 0 & 0 & 0 \\
  0 & 0 & 0 & 0 & 0 & 0 & 0 & 0 & 0 & 0 & 0 & \zeta_{32}^{7} \\
  0 & 0 & 0 & 0 & 0 & 0 & 0 & 0 & 0 & 0 & \zeta_{32}^{7} & 0
\end{matrix}\right)$$

and
$$M_{(8)} := \left(\begin{matrix}
  0 & 0 & 0 & 0 & \frac{1}{\sqrt{2}} & 0 & \frac{1}{\sqrt{2}} & 0 & 0 & 0 & 0 & 0 \\
  0 & 0 & 0 & 0 & 0 & 0 & 0 & 0 & \frac{1}{\sqrt{2}} & 0 & \frac{1}{\sqrt{2}} & 0 \\
  0 & 0 & 0 & 0 & \frac{1}{\sqrt{2}} & 0 & -\frac{1}{\sqrt{2}} & 0 & 0 & 0 & 0 & 0 \\
  0 & 0 & 0 & 0 & 0 & 0 & 0 & 0 & -\frac{1}{\sqrt{2}} & 0 & \frac{1}{\sqrt{2}} & 0 \\
  \frac{1}{\sqrt{2}} & 0 & \frac{1}{\sqrt{2}} & 0 & 0 & 0 & 0 & 0 & 0 & 0 & 0 & 0 \\
  0 & 0 & 0 & 0 & 0 & 0 & 0 & 0 & 0 & \frac{\sqrt{2-\sqrt{2}}}{2} & 0 & \frac{\sqrt{2+\sqrt{2}}}{2} \\
  \frac{1}{\sqrt{2}} & 0 & -\frac{1}{\sqrt{2}} & 0 & 0 & 0 & 0 & 0 & 0 & 0 & 0 & 0 \\
  0 & 0 & 0 & 0 & 0 & 0 & 0 & 0 & 0 & \frac{\sqrt{2+\sqrt{2}}}{2} & 0 & -\frac{\sqrt{2-\sqrt{2}}}{2} \\
  0 & \frac{1}{\sqrt{2}} & 0 & -\frac{1}{\sqrt{2}} & 0 & 0 & 0 & 0 & 0 & 0 & 0 & 0 \\
  0 & 0 & 0 & 0 & 0 & \frac{\sqrt{2-\sqrt{2}}}{2} & 0 & \frac{\sqrt{2+\sqrt{2}}}{2} & 0 & 0 & 0 & 0 \\
  0 & \frac{1}{\sqrt{2}} & 0 & \frac{1}{\sqrt{2}} & 0 & 0 & 0 & 0 & 0 & 0 & 0 & 0 \\
  0 & 0 & 0 & 0 & 0 & \frac{\sqrt{2+\sqrt{2}}}{2} & 0 & -\frac{\sqrt{2-\sqrt{2}}}{2} & 0 & 0 & 0 & 0
\end{matrix}\right).$$
\end{thm}

\begin{cor} \label{corollary xi-image order 8}
We have $\xi_{1/2}(H_{(8)})(\tau)=-\frac{1}{2} \ g_{(8)}(\tau)$.
\end{cor}

In the following we consider the congruence subgroup
$$\Gamma(8) = \left\{\left(\begin{matrix}
  a & b \\
  c & d \\
\end{matrix}\right) \in \text{SL}_{2}(\mathbb{Z}) \ \Big| \ b \equiv c \equiv 0 \ (8), \ a \equiv d \equiv 1 \ (8) \right\}.$$

This leads to:

\begin{thm} \label{theorem order 8 scalar valued harmonic Maass forms}
The components of the vector valued harmonic weak Maass form $H_{(8)}$ are scalar valued harmonic weak Maass forms of weight $\nicefrac{1}{2}$ for the subgroup
$$\{(\gamma,\phi) \in \emph{\text{Mp}}_{2}(\mathbb{Z}) \ | \ \gamma \in \Gamma(8) \}$$

of the metaplectic group $\emph{\text{Mp}}_{2}(\mathbb{Z})$.
\end{thm}

\begin{rem}
As before, the $\xi$-images of the harmonic weak Maass forms in Theorem \ref{theorem order 8 scalar valued harmonic Maass forms} can be directly obtained from Corollary \ref{corollary xi-image order 8}.
\end{rem}

Finally we consider the yet omitted mock theta functions $U_{0}$ and $U_{1}$. Using their relations to $S_{0},S_{1},T_{0}$ and $T_{1}$ in Lemma \ref{lemma connection between eighth order U_0, S_0 and S_1 and U_1, T_0 and T_1} and denoting the components of $H_{(8)}$ by  $h_{(8),0},\dots,h_{(8),11}$ gives us
\begin{flalign*}
h_{(8),4}(\tau) \pm h_{(8),6}(\tau) &= q^{-\frac{1}{32}} \ U_{0}( \pm q^{\frac{1}{4}}) + \frac{i}{4} \ \int_{-\overline{\tau}}^{i \infty} \frac{\theta_{8,1}(z) \mp \theta_{8,3}(z) \pm \theta_{8,5}(z)-\theta_{8,7}(z)}{\sqrt{-i(z+\tau)}} \ dz, \\
h_{(8),8}(\tau) \pm h_{(8),10}(\tau) &= q^{-\frac{1}{32}} \ U_{1}(\pm q^{\frac{1}{4}}) + \frac{i}{8} \ \int_{-\overline{\tau}}^{i \infty} \frac{-\theta_{8,1}(z) \pm \theta_{8,3}(z) \mp \theta_{8,5}(z)+\theta_{8,7}(z)}{\sqrt{-i(z+\tau)}} \ dz.
\end{flalign*}

It can be shown via Sage \cite{sagemath} that $h_{(8),4}$ and $h_{(8),6}$ have the same transformation behaviour under all generators of $\Gamma(8)$, and also the two functions $h_{(8),8}$ and $h_{(8),10}$ have the same transformation properties under all generators of $\Gamma(8)$. From this and Theorem \ref{theorem order 8 scalar valued harmonic Maass forms} we can conclude:

\begin{thm} \label{theorem additional order 8 scalar valued harmonic Maass forms}
The functions $h_{(8),4} \pm h_{(8),6}$ and $h_{(8),8} \pm h_{(8),10}$ are scalar valued harmonic weak Maass forms of weight $\nicefrac{1}{2}$ for the subgroup
$$\{(\gamma,\phi) \in \emph{\text{Mp}}_{2}(\mathbb{Z}) \ | \ \gamma \in \Gamma(8) \}$$

of the metaplectic group $\emph{\text{Mp}}_{2}(\mathbb{Z})$.
\end{thm}

With the treatment of $U_{0}$ and $U_{1}$ we have now related all eighth order mock theta functions to scalar valued harmonic weak Maass forms.

\begin{rem}
We get the $\xi$-images of the harmonic weak Maass forms in Theorem \ref{theorem additional order 8 scalar valued harmonic Maass forms} from Corollary \ref{corollary xi-image order 8} by adding and subtracting the respective components of $\xi_{1/2}(H_{(8)})(\tau)$.
\end{rem}

\subsection{Mock Theta Functions of order $10$} \label{section mock theta functions of
order 10}
The necessary completion and its transformation behaviour has already been studied
by Moore in \cite{moore2012modular}. We consider the matrices
\begin{align}\label{N10}
N_{(10)}:=\begin{pmatrix}
0&0&\zeta_{10}&0&0&0\\0&0&0&\zeta_{10}^{-1}&0&0\\\zeta_{10}&0&0&0&0&0\\0&\zeta_{10}^{-1}&0&0&0&0\\0&0&0&0&\zeta_{40}^{-1}&0\\0&0&0&0&0&\zeta_{40}^{-9}
\end{pmatrix}
\end{align}
and
\begin{align}\label{M10}
M_{(10)}:=\begin{pmatrix}
0&0&0&0&\sin(\frac{2\pi}{5})&-\sin(\frac{\pi}{5})\\0&0&0&0&\sin(\frac{\pi}{5})&\sin(\frac{2\pi}{5})\\0&0&\sin(\frac{2\pi}{5})&\sin(\frac{\pi}{5})&0&0\\0&0&\sin(\frac{\pi}{5})&-\sin(\frac{2\pi}{5})&0&0\\\sin(\frac{2\pi}{5})&\sin(\frac{\pi}{5})&0&0&0&0\\-\sin(\frac{\pi}{5})&\sin(\frac{2\pi}{5})&0&0&0&0
\end{pmatrix}.
\end{align}
\begin{thm}[\cite{moore2012modular}, Theorem 1]
The vector
$$F_{(10)}(\tau)=\left(\begin{matrix}q^{\frac{1}{10}}\ \phi(q^{\frac{1}{2}})\\ q^{-\frac{1}{10}}\ \psi(q^{\frac{1}{2}})\\ q^{\frac{1}{10}}\ \phi(-q^{\frac{1}{2}})\\ q^{-\frac{1}{10}}\ \psi(-q^{\frac{1}{2}})\\ q^{-\frac{1}{40}}\ X(q)\\ q^{-\frac{9}{40}}\ \chi(q) \end{matrix}\right)$$
is the holomorphic part of $H_{(10)}=(h_{(10),0},h_{(10),1},h_{(10),2},h_{(10),3},h_{(10),4},h_{(10),5})^T\in H_{1/2}^+$, which is a harmonic weak Maass form of weight $\nicefrac{1}{2}$, transforming as
\begin{align*}
H_{(10)}(\tau+1)=N_{(10)} \ H_{(10)}(\tau)
\end{align*}
and
\begin{align*}
H_{(10)}\bigg(-\frac{1}{\tau}\bigg)=\sqrt{-i\tau} \ \frac{2}{\sqrt{5}} \ M_{(10)} \ H_{(10)}(\tau),
\end{align*}
where the matrices $N_{(10)}$ and $M_{(10)}$ are defined as in (\ref{N10}) and (\ref{M10}).
\end{thm}
The following result is a simple consequence from the statement above.
\begin{lemma}
The function
\begin{align*}
\widetilde{H}_{(10)}:&=(h_{(10),0}+h_{(10),2}) \ [-\mathfrak{e}_6+\mathfrak{e}_{-6}]+(h_{(10),0}-h_{(10),2}) \ [-\mathfrak{e}_4+\mathfrak{e}_{-4}] \\
&+(h_{(10),1}+h_{(10),3}) \ [-\mathfrak{e}_2+\mathfrak{e}_{-2}]+(h_{(10),1}-h_{(10),3}) \ [-\mathfrak{e}_8+\mathfrak{e}_{-8}] \\
&+h_{(10),4} \ [\mathfrak{e}_1-\mathfrak{e}_{-1}-\mathfrak{e}_9+\mathfrak{e}_{-9}]+h_{(10),5} \ [\mathfrak{e}_3-\mathfrak{e}_{-3}-\mathfrak{e}_7+\mathfrak{e}_{-7}]
\end{align*}
transforms with respect to the dual Weil representation $\overline{\rho}_L$ of weight $\nicefrac{1}{2}$ and level $N=10$.
\end{lemma}
\begin{defin}
We define the function
\begin{align}\label{e10}
e_{(10)}(z):=\bigg(\frac{\eta(z)\eta(2z)}{\eta(5z)\eta(10z)}\bigg)^2-25\bigg(\frac{\eta(5z)\eta(10z)}{\eta(z)\eta(2z)}\bigg)^2.
\end{align}
\end{defin}
This function is a weakly holomorphic modular form of weight $0$, level $10$ whose principal part
starts with $q^{-1}$.
\begin{thm}
Let $e_{(10)}(z)\in M_0^!(10)$ be defined as in (\ref{e10}).
\begin{enumerate}
\item[(1)] For $n\geq1$ the coefficients $a_X(n)$ of $X(q)$ are given by
\begin{align*}
a_X(n)&=\frac{i}{2\sqrt{40n-1}}\big(\textup{tr}_{e_{(10)}}^+(1-40n,1)-\textup{tr}_{e_{(10)}}^-(1-40n,1)
\big).
\end{align*}
\item[(2)] For $n\geq1$ the coefficients $a_\chi(n)$ of $\chi(q)$ are given by
\begin{align*}
a_\chi(n)&=\frac{i}{2\sqrt{40n-9}}\big(\textup{tr}_{e_{(10)}}^+(9-40n,3)-\textup{tr}_{e_{(10)}}^-(9-40n,3)
\big).
\end{align*}
\item[(3)] For $n\geq1$ the coefficients $a_\phi(n)$ of $\phi(q)$ are given by
\begin{align*}
a_\phi(n)&=\begin{cases}\frac{-i}{4\sqrt{40\frac{n+2}{2}-36}}\big(\textup{tr}_{e_{(10)}}^+(36-40\frac{n+2}{2},6)-\textup{tr}_{e_{(10)}}^-(36-40\frac{n+2}{2},6)
\big),&\text{ if }n\text{
is even,}\\\frac{-i}{4\sqrt{40\frac{n+1}{2}-16}}\big(\textup{tr}_{e_{(10)}}^+(16-40\frac{n+1}{2},4)-\textup{tr}_{e_{(10)}}^-(16-40\frac{n+1}{2},4)
\big),&\text{ if }n\text{ is odd.} \end{cases}
\end{align*}
\item[(4)] For $n\geq1$ the coefficients $a_\psi(n)$ of $\psi(q)$ are given by
\begin{align*}
a_\psi(n)&=\begin{cases}\frac{-i}{4\sqrt{40\frac{n}{2}-4}}\big(\textup{tr}_{e_{(10)}}^+(4-40\frac{n}{2},2)-\textup{tr}_{e_{(10)}}^-(4-40\frac{n}{2},2)
\big),&\text{ if }n\text{
is even,}\\\frac{-i}{4\sqrt{40\frac{n+3}{2}-64}}\big(\textup{tr}_{e_{(10)}}^+(64-40\frac{n+3}{2},8)-\textup{tr}_{e_{(10)}}^-(64-40\frac{n+3}{2},8)
\big),&\text{ if }n\text{ is odd.} \end{cases}
\end{align*}
\end{enumerate}
\end{thm}

\newpage

\nocite{*}

\end{document}